\documentclass[11pt,reqno]{amsart}

\usepackage[text={130mm,210mm},centering]{geometry}   
\usepackage[utf8]{inputenc}
\usepackage[english]{babel}
\usepackage{amsmath}
\usepackage{amsfonts}
\usepackage{amssymb}
\usepackage{epstopdf}
\usepackage[all]{xy}
\usepackage{amsmath,amsthm}
\usepackage{amsmath,amssymb}
\usepackage{amsmath,chemarrow}
\usepackage{appendix}

\usepackage{latexsym}
\usepackage{tikz}
\usepackage{tikz-cd}
\usepackage{amsmath,amsthm}

\usepackage[mathscr]{eucal}
\usepackage[all]{xy}
\usepackage{mathrsfs}
\usepackage{hyperref}
\usepackage{pifont}

\allowdisplaybreaks

\newtheorem{thm}{Theorem}[section]
\newtheorem{cor}[thm]{Corollary}
\newtheorem{lem}[thm]{Lemma}
\newtheorem{prop}[thm]{Proposition}
\newtheorem{ques}[thm]{Question}
\theoremstyle{definition}
\newtheorem{defn}[thm]{Definition}
\newtheorem{rem}[thm]{Remark}
\newtheorem{examp}[thm]{Example}

\numberwithin{equation}{section}

\newcommand{\sO}{{\mathcal O}}

\newcommand{\G}{{\mathbb G}}

\renewcommand{\P}{{\mathbb P}}

\newcommand{\Spec}{\mathrm{Spec}}

\newcommand{\rk}{\mathrm{rk}}

\newcommand{\Fil}{\mathrm{Fil}}
\newcommand{\Gr}{\mathrm{Gr}}

\author[Mao Sheng]{Mao Sheng}
\author[Hao Sun]{Hao Sun}
\author[Jianping Wang]{Jianping Wang}

\begin{document}
\title[Existence of gr-semistable filtrations]{On the existence of gr-semistable filtrations of orthogonal/symplectic $\lambda$-connections}
\maketitle

\begin{abstract}
In this paper, we study the existence of gr-semistable filtrations of orthogonal/symplectic $\lambda$-connections. It is known that gr-semistable filtrations always exist for flat bundles in arbitrary characteristic. However, we found a counterexample of orthogonal flat bundles of rank 5 in positive characteristic. The central new idea in this example is the notion of quasi gr-semistability for orthogonal/symplectic $\lambda$-connections. We establish the equivalence between gr-semistability and quasi gr-semistablity for an orthogonal/symplectic $\lambda$-connection. This provides a way to determine whether an orthogonal/symplectic $\lambda$-connection is gr-semistable. As an application, we obtain a characterization of gr-semistable orthogonal $\lambda$-connections of rank $\leq 6$.
\end{abstract}

\tableofcontents

\renewcommand{\thefootnote}{\fnsymbol{footnote}}

\section{Introduction}

Let $X$ be a smooth projective variety over an algebraically closed field $k$. Given $\lambda \in k$, a $\lambda$-connection on a vector bundle $V$ over $X$ is a $k$-linear morphism
\begin{align*}
    \nabla: V \rightarrow V \otimes \Omega_X,
\end{align*}
satisfying the so-called $\lambda$-Leibniz rule. It is integrable if it satisfies the integrability condition $\nabla^2=0$. When $\lambda=1$, $(V,\nabla)$ is a flat bundle, while when $\lambda=0$, the pair is a Higgs bundle. When $k=\mathbb{C}$, Simpson considered the moduli space $\mathcal{M}_{\rm Hod}(X,r)$ of (Gieseker) semistable integrable $\lambda$-connections on $X$ of rank $r$, and showed that there is a natural morphism
\begin{align*}
    \mathcal{M}_{\rm Hod}(X,r) \rightarrow \mathbb{A}^1, \quad (V,\nabla,\lambda) \mapsto \lambda
\end{align*}
such that the preimage of $\lambda=0$ is the Dolbeault moduli space $\mathcal{M}_{\rm Dol}(X,r)$ and the preimage of $\lambda=1$ is the de Rham moduli space $\mathcal{M}_{dR}(X,r)$ \cite[Proposition 4.1]{Si97}. There is a natural $\mathbb{G}_m$-action on $\mathcal{M}_{\rm Hod}(X,r)$, i.e.
\begin{align*}
    t \cdot (V,\nabla,\lambda) := (V, t\nabla, t\lambda).
\end{align*}
Consider the limit $\lim\limits_{t \rightarrow 0} (V, t\nabla, t\lambda)$. In the case of curves (hence Gieseker semistability coincides with slope semistability), Simpson gave an effective approach to find such a limit, which is called \emph{iterated destabilizing modifications} \cite{S10}. We briefly review the  background of this approach. Let $(V,\nabla)$ be a flat bundle. An effective finite decreasing filtration of subbundles
\begin{align*}
\Fil: 0=Fil_m \subset Fil_{m-1}  \subset \cdots \subset Fil_0=V
\end{align*}
is called a \emph{Hodge filtration} of $(V,\nabla)$ if $\Fil$ satisfies the \emph{Griffiths transversality condition}:
$$\nabla(Fil_{i+1})\subset Fil_{i}\otimes \Omega_X, \ i \geq 1.$$
From a Hodge filtration we shall obtain a well-defined Higgs bundle $(E,\theta)$ by taking grading, where $E: = \Gr_{\Fil}(V)$ and $\theta$ is induced by $\nabla$. If $(E,\theta)$ is semistable as a Higgs bundle, then we say the Hodge filtration is \emph{gr-semistable}. In fact, Simpson proved that gr-semistable filtrations always exist. The notion of a gr-semistable Hodge filtration (with respect to a chosen polarization) has been generalized to a high dimensional base (\cite{LSZ19,L13}). Here we shall make the existence of a gr-semistable Hodge filtration as a defining property of a $\lambda$-connection, which is termed as \emph{gr-semistability} (see Definition \ref{gr-semistable}). It has the following fundamental significance.

\begin{thm}[Simpson (Theorem 2.5 \cite{S10}), Lan-Sheng-Yang-Zuo (Theorem A.4 \cite{LSZ19}), Langer (Theorem 5.5 \cite{L13})]\label{the case of flat bundles}
Let $k$ be an algebraically closed field and $X$ a smooth projective variety over $k$, equipped with an ample line bundle $L$. Let $(V,\nabla)$ be an integrable $\lambda$-connection over $X$. Then $(V,\nabla)$ is $\mu_L$-semistable if and only if it is $\mu_L$-gr-semistable.
\end{thm}

When $L$ is clear in the context, we shall omit $\mu_L$ when we are speaking about slope semistability. In fact, we shall simply speak of semistability as we shall only deal with slope semistability throughout the paper. The above theorem provides an alternative way to understand semistability for flat bundles in positive characteristic. Moreover, the result plays a crucial role in the theory of semistable Higgs-de Rham flows \cite{LSZ19} and then its subsequent applications.

In our attempt to generalize the theory of semistable Higgs-de Rham flows to the case of principal bundles, we found a different phenomenon, as we shall explain below. On the other hand, our study has also been driven by the following question of Simpson, which he proposed after he also proved the existence of the limit of $\G_m$-action in the case of principal bundles (\cite[Corollary 8.3]{S10}).
\begin{ques}[\cite{S10} Question 8.4]\label{quest_simp}
How to give an explicit description of the limiting points in terms of Griffiths-transverse parabolic reductions in the case of principal bundles?
\end{ques}
He proposed a principal bundle approach to the instability flag. In this paper, we shall study this question, together with his proposal, in the special case of orthogonal/symplectic bundles.

The following example indicates that the approach of iterated destabilizing modifications may not work any more in the case of principal bundles, and hence some new idea (perhaps some sort of base change) is requested.
\begin{examp}
 Let $X$ be a genus $g$ curve over an algebraically closed field $k$ in positive characteristic $p$ with $2<p < g-1$. Let $M'$ be a stable bundle on $X$ with $\rk(M')=2$ and $\deg(M')=1$ whose Frobenius pullback $M:=F^*_X(M')$ is again stable. By Riemann-Roch theorem, the bundle $(\wedge^2 M^{\vee}) \otimes \Omega_X$ has nonzero global sections. Now the orthogonal structure $\langle \, , \, \rangle$ on
\begin{align*}
    V = M \oplus M^{\vee} \oplus \mathcal{O}_X,
\end{align*}
is defined as the orthogonal direct sum of the standard orthogonal structure $\langle \, , \, \rangle_M$ over $M \oplus M^\vee$ and the natural one $\langle \, , \, \rangle_{\mathcal{O}_X}$ on $\mathcal{O}_X$. Then $V$ is an unstable orthogonal bundle of $\rk(V)=5$ with maximal destabilizer $M$. Pick any nonzero section of $(\wedge^2 M^{\vee}) \otimes \Omega_X\subset (M^{\vee})^{\otimes 2}\otimes \Omega_X$, which shall be regarded as a nonzero morphism $\theta_1: M \rightarrow M^{\vee} \otimes \Omega_X$. Clearly for $a,b\in M$, it holds that
$$\langle \theta_1(a),b \rangle_M + \langle a,\theta_1(b) \rangle_M = 0.$$
Therefore, we obtain an orthogonal connection on $V$
\begin{align*}
\nabla:=
\begin{pmatrix}
    \nabla_0 & & \\
    \theta_1 & -\nabla^\vee_0 & \\
    & & d
\end{pmatrix}: V \rightarrow V \otimes \Omega_X,
\end{align*}
where $\nabla_0$ is the canonical connection on $M$ and $d: \mathcal{O}_X \rightarrow \Omega_X$ is the exterior differential. Then the triple $(V, \langle \, , \, \rangle, \nabla)$ is a semistable orthogonal integrable connection, but does not admit \emph{any} orthogonal Hodge filtration so that its associated graded orthogonal Higgs bundle is semistable. The statement is based on our next result (see \S\ref{subsect_rak_5} for a complete proof of this statement).
\end{examp}

The example shows that the analogue of Theorem \ref{the case of flat bundles} is \emph{false} in the case of orthogonal integrable connections. That is, the notion of gr-semistability is strictly stronger than that of semistabililty for orthogonal integrable connections (at least in positive characteristic). To the positive side of this study, we obtained the following theorem, which is the main result of the paper.
\begin{thm}(Theorem \ref{thm_gr_and_quasi_gr_odd} and \ref{thm_even})
Let $k,X,L$ be as in Theorem \ref{the case of flat bundles}. Let $D$ be a reduced normal crossing divisor on $X$. Let $(V,\langle,\rangle,\nabla)$ be a logarithmic orthogonal/symplectic $\lambda$-connection. Then it is gr-semistable if and only if it is semistable and quasi gr-semistable.
\end{thm}
\begin{rem}
We also consider $\Omega_X(\log D)$-valued $\lambda$-connections, which is termed as logarithmic $\lambda$-connections. Note that the integrability condition on a $\lambda$-connection is not required in the statement.
\end{rem}

The definition of quasi gr-semistability is somewhat technical. See Definition \ref{defn_gr_semi_odd} and \ref{defn_gr_semi_even} for a precise formulation. Roughly speaking, the quasi gr-semistability is a necessary and sufficient condition, under which the method of iterated destabilizing modification works.

The structure of the paper is manifest: After laying down the necessary setup and terminologies in \S\ref{sect_pre}, we establish the main result in \S\ref{sect_odd} and \S\ref{sect_even}. In \S\ref{sect_small_rank}, we explicate the quasi gr-semistability in small ranks, viz. $\rk(V)\leq 6$.

\section{Preliminaries}\label{sect_pre}

Let $X$ be a smooth projective variety over an algebraically closed field $k$ with a fixed reduced normal crossing divisor $D=\sum\limits_{i=1}^{k}D_i$. Denote by $\Omega_X$ the cotangent sheaf of $X$. In \S\ref{subsect_rees}, we review the correspondence between Rees $G$-bundles and filtered $G$-bundles studied in \cite{Lo17a,Lo17b}. The correspondence gives a way to study filtered $G$-bundles by taking an appropriate faithful representation $G \rightarrow {\rm GL}(W)$ and consider the corresponding filtered bundles. Based on this idea, we give the definition of filtered orthogonal/symplectic sheaves, which is also considered in \cite[\S5]{GS03}. Then, in \S\ref{subsect_log_lambda_conn}, we introduce the main objects studied in this paper, the \emph{Hodge filtered logarithmic orthogonal/symplectic $\lambda$-connections} $(V, \langle \, , \, \rangle, \nabla, \Fil)$ (Definition \ref{defn_na_Hod_fil}), where $(V, \langle \, , \, \rangle)$ is an orthogonal/symplectic sheaf, $\nabla$ is an orthogonal/symplectic $\lambda$-connection and $\Fil$ is an orthogonal/symplectic filtration obeying Griffiths transversality. In the end of this section, we discuss the stability condition of $(V, \langle \, , \, \rangle, \nabla)$ and its Harder--Narasimhan filtration.

\subsection{Rees Construction}\label{subsect_rees}
In this subsection, we briefly review Rees construction, which gives a correspondence between filtered $G$-bundles on $X$ and $\mathbb{G}_m$-equivariant $G$-bundles on $X \times \mathbb{A}^1$. We refer the reader to \cite{L13,Lo17a,Lo17b} for more details.

Let $V$ be a vector bundle (locally free sheaf) of finite rank on $X$ together with a filtration
\begin{align*}
    \Fil: 0 = Fil_m \subseteq Fil_{m-1} \subseteq \dots \subseteq Fil_0 = V
\end{align*}
of subbundles on $X$. Furthermore, we always assume that $Fil_i =0$ for $i \geq m$ and $Fil_i = V$ for $i \leq 0$. Such a pair $(V,\Fil)$ is called a \emph{filtered bundle} on $X$.

Now we consider Rees bundles. A \emph{Rees bundle} on $X$ is a $\mathbb{G}_m$-equivariant bundle $\mathcal{V}$ of finite rank on $X \times \mathbb{A}^1$ (with respect to the natural $\mathbb{G}_m$-action on $\mathbb{A}^1$). Let $\mathbb{A}^1 = \Spec \, k[t]$. As proven in \cite[Lemma 2.1.3]{Lo17b}, a Rees bundle on $X$ is equivalent to a triple $(V,\mathcal{V},\varphi)$ such that
\begin{itemize}
    \item $\mathcal{V}$ is a vector bundle over $X \times \mathbb{A}^1$,
    \item $V$ is a vector bundle over $X$,
    \item $\varphi: \mathcal{V}|_{t \neq 0}  \xrightarrow{\cong} \pi_X^* V$ is an isomorphism, where $\pi_X$ is the projection to $X$, such that there exist generators $\widetilde{v}_i$ for $\mathcal{V}$ and $v_i$ for $V$ satisfying $\varphi(\widetilde{v}_i) = v_i \otimes t^{n_i}$, where $n_i \in \mathbb{Z}$.
\end{itemize}
Applying the argument of \cite[Proposition 2.1.5]{Lo17b}, we have the following equivalence.

\begin{lem}\label{lem_rees_fil}
The category ${\rm Fil}_X$ of filtered bundles on $X$ is equivalent to the category ${\rm Rees}_X$ of Rees bundles on $X$. Moreover, the equivalence holds as tensor categories.
\end{lem}

\begin{proof}
We only give the correspondence briefly. Given a filtered bundle $(V, {\rm Fil})$, we define a bundle
\begin{align*}
    \mathcal{V}: = \sum t^{-i} Fil_{i} \otimes_{\mathcal{O}_X} \mathcal{O}_{X \times \mathbb{A}^1}
\end{align*}
on $X \times \mathbb{A}^1$ together with a natural $\mathbb{G}_m$-equivariant structure induced from that on $\mathbb{A}^1$. Clearly, $\mathcal{V}$ is a Rees bundle on $X$.

On the other direction, given a Rees bundle $\mathcal{V}$ on $X$, denote by $V: = \mathcal{V}|_{t = 1}$ the restriction to $t=1$. We will give a filtration on $V$ by working locally on an affine open subset $U \subseteq X$, over which $V(U)$ is a free $\mathcal{O}_X$-module. Let $\{v_i\}$ be a set of generators of $V(U)$ and let $\{\widetilde{v}_i\}$ be a set of generators of $\mathcal{V}|_{t \neq 0} (U \times \mathbb{G}_m)$ such that $\varphi(\widetilde{v}_i) = v_i \otimes t^{n_i}$ under the isomorphism $\varphi: \mathcal{V}|_{t \neq 0} \cong \pi^*_X V$. Since $V(U)$ is free, we can assume that $\{v_i\}$ is a basis of $V(U)$. Then, given an integer $i$, we define a free submodule $Fil_i(U) \subseteq V(U)$, which is spanned by $v_j$ such that $n_j \leq -i$. In other words, $Fil_i(U)$ includes all sections $v \in V(U)$ such that the vanishing order of $\widetilde{v}$ at $t=0$ is at least $i$. Clearly, $Fil_{i+1}(U) \subseteq Fil_{i}(U)$ and we obtain a filtration of $V(U)$. By patching together $Fil_i(U)$, we obtain a subbundle $Fil_i \subseteq V$, and thus a filtration of subbundles of $V$ as desired.
\end{proof}

\begin{rem}\label{rem_graded}
Given a filtered bundle $(V, \Fil)$, we can define a graded bundle ${\rm Gr}_{\rm Fil}(V) := \bigoplus\limits_{i=0}^{m-1} \frac{Fil_i}{Fil_{i+1}}$, which is also the restriction of the corresponding Rees bundle $\mathcal{V}$ to $t= 0$.
\end{rem}

Now we consider the case of principal bundles. Let $G$ be a connected reductive group over $k$. Similar to the case of vector bundles, a \emph{Rees $G$-bundle} is defined as a $\mathbb{G}_m$-equivariant $G$-bundle on $X \times \mathbb{A}^1$. Now we will give an equivalent definition of Rees $G$-bundle in the viewpoint of Tannakian categories. Denote by ${\rm Rep}(G)$ the category of $G$-representations with values in free $k$-modules of finite rank, which is a (rigid) tensor category. Recall that a $G$-bundle is equivalent to a faithful exact tensor functor ${\rm Rep}(G) \rightarrow {\rm Vect}_X$ (see \cite[\S 6]{Si92} for instance). Now we choose a representation $\rho : G \rightarrow {\rm GL}(W)$ and a Rees $G$-bundle $\mathcal{E}$, and define the associated bundle $\mathcal{E} \times_G W$, which is a Rees bundle on $X$. This construction induces a fiber functor ${\rm Rep}(G) \rightarrow {\rm Rees}_X$.  Lovering proved the following equivalence of categories.
\begin{lem}[Proposition 2.2.7 in \cite{Lo17b}]\label{lem_rees_G}
The category ${\rm Rees}^G_X$ of Rees $G$-bundles is equivalent to the category ${\rm Fib}({\rm Rep}(G), {\rm Rees}_X)$ of fiber functors ${\rm Rep}(G) \rightarrow {\rm Rees}_X$.
\end{lem}

For a \emph{filtered $G$-bundle}, it is defined as a pair $(E,\sigma)$, where $E$ is a $G$-bundle and $\sigma: X \rightarrow E/P$ is reduction of structure group for some parabolic subgroup $P \subseteq G$. Similar to Lemma \ref{lem_rees_G}, we have the following equivalent description:
\begin{lem}[\S3 in \cite{Lo17a}]\label{lem_fil_G}
The category ${\rm Fil}_X^G$ of filtered $G$-bundles is equivalent to the category ${\rm Fib}({\rm Rep}(G), {\rm Fil}_X)$ of fiber functors ${\rm Rep}(G) \rightarrow {\rm Fil}_X$.
\end{lem}

Combining Lemma \ref{lem_rees_fil}, \ref{lem_rees_G} and \ref{lem_fil_G} altogether
\begin{center}
\begin{tikzcd}
{\rm Rees}^G_X \arrow[rr, leftrightarrow, "\text{Lemma \ref{lem_rees_G}}"] \arrow[dd, dotted, leftrightarrow] & & {\rm Fib}({\rm Rep}(G), {\rm Rees}_X) \arrow[dd, leftrightarrow, "\text{Lemma \ref{lem_rees_fil}}"] \\
& & \\
{\rm Fil}^G_X \arrow[rr, leftrightarrow, "\text{Lemma \ref{lem_fil_G}}"] & & {\rm Fib}({\rm Rep}(G), {\rm Fil}_X) \ ,
\end{tikzcd}
\end{center}
we see that the category ${\rm Rees}^G_X$ of Rees $G$-bundles is equivalent to the category ${\rm Fil}^G_X$ of filtered $G$-bundles. The approach of Tannakian categories not only gives an equivalent description of $G$-bundles, Rees $G$-bundles and filtered $G$-bundles theoretically, but also provides an effective way to work on specific groups.

Now we take $G$ to be orthogonal/symplectic groups, and
denote by $G \hookrightarrow {\rm GL}(W)$ the standard representation. Therefore, an orthogonal/symplectic bundle $E$ is equivalent to a pair $(V, \langle \, , \, \rangle)$, where $V:=E \times_G W$ is the associated bundle and $\langle \, , \, \rangle$ is a nondegenerate symmetric/skew-symmetric bilinear form. Since the base variety $X$ is of arbitrary dimension, it is more flexible to work with orthogonal/symplectic sheaf (see also \cite[\S 5]{GS03}).
\begin{defn}
An \emph{orthogonal/symplectic sheaf} is a pair $(V, \langle \, , \, \rangle)$ such that
\begin{enumerate}
\item $V$ is a torsion free coherent sheaf over $X$,

\item $\langle \, , \, \rangle:V\otimes V\longrightarrow \sO_X$ is a symmetric/skew-symmetric $\sO_X$-bilinear form and nondegenerate on $U \subseteq X$, where $U$ is an open subset contained in the locally free locus of $V$ with $\textrm{codim}(X \backslash U)\geq 2$.

\item $\det(V)^2 \cong \mathcal{O}_X$.
\end{enumerate}
\end{defn}
The pairing induces an $\sO_X$-linear morphism $\langle \, , \, \rangle: V\to V^{\vee}$, which is an isomorphism over $U$.
\begin{defn}
Let $(V,\langle \, , \, \rangle)$ be an orthogonal/symplectic sheaf and let $F\subset V$ be a saturated subsheaf.
\begin{enumerate}
\item We call $F^{\perp}:=\ker(V\xrightarrow{\langle \, , \, \rangle} V^{\vee}\rightarrow F^{\vee})$ the \emph{orthogonal complement} of $F$.

\item  $F$ is \emph{isotropic} if $\langle \, , \, \rangle|_{F\otimes F}=0.$ It is easy to check that $F$ is isotropic if and only if $F\subset F^{\perp}$.
\end{enumerate}
\end{defn}

A filtered orthogonal/symplectic bundle is a pair $(E,\sigma)$, where $\sigma: X \rightarrow E/P$ is a reduction of structure group. By Lemma \ref{lem_fil_G}, it is equivalent to a fiber functor ${\rm Rep}(G) \rightarrow {\rm Fil}_X$. Choosing the standard representation $G \hookrightarrow {\rm GL}(W)$, reductions of structure group of $E$ correspond to filtrations of isotropic subbundles of $V:=E \times_G W$ and we refer the reader to \cite[\S 6]{CS21} and \cite[\S 5]{Ya22} for more details. Then a sheaf version is given as follows.

\begin{defn}
An \emph{orthogonal/symplectic filtration} of an orthogonal/symplectic sheaf $(V,\langle \, , \, \rangle)$ is a filtration of saturated subsheaves
\begin{align*}
\Fil: 0=Fil_m \subset Fil_{m-1}  \subset \cdots \subset Fil_0=V
\end{align*}
such that $Fil_{i}= Fil_{m-i}^\perp$. Furthermore, a filtration of saturated subsheaves will be called a \emph{saturated filtration} for simplicity.
\end{defn}

\begin{rem}
Let $(V, \langle \, , \, \rangle)$ be an orthogonal/symplectic sheaf together with an orthogonal/symplectic filtration ${\rm Fil}$. There is a natural orthogonal/symplectic structure on the graded bundle
$\Gr_{\Fil}(V):=\mathop{\bigoplus}\limits_{i=0}^{m-1}\frac{Fil_i} {Fil_{i+1}}$ induced by that on
\begin{align*}
    \frac{Fil_i}{Fil_{i+1}} \oplus \frac{Fil_{m-i-1}}{Fil_{m-i}} \cong \frac{Fil_{i}}{Fil_{i+1}} \oplus \frac{Fil^\perp_{i+1}}{Fil^\perp_{i}},
\end{align*}
which is denoted by $\overline{\langle \, , \, \rangle}$. Then, we obtain a graded orthogonal/symplectic sheaf $(\Gr_{\Fil}(V) , \overline{\langle \, , \, \rangle})$. Moreover, in the case of orthogonal/symplectic bundles, the graded orthogonal/symplectic bundle $({\rm Gr}_{\Fil}(V), \overline{\langle \, , \, \rangle})$ we obtain is exactly the restriction of the corresponding orthogonal/symplectic Rees bundle to $t=0$ as given in Remark \ref{rem_graded}.
\end{rem}

\subsection{Logarithmic Orthogonal/Symplectic $\lambda$-connections}\label{subsect_log_lambda_conn}

\begin{defn}\label{defn_lambda_conn}
Let $(V,\langle \, , \, \rangle)$ be an orthogonal/symplectic sheaf on $X$. Let $\lambda\in k$. A \emph{logarithmic orthogonal/symplectic $\lambda$-connection} is a $k$-linear map $\nabla: V \rightarrow V\otimes \Omega_X(\log D)$ such that on each open subset $U \subseteq X$, it satisfies
\begin{enumerate}
\item \textbf{Leibniz rule}:
\begin{align*}
    \nabla (fa)=\lambda a\otimes df+f\nabla(a),
\end{align*}
where $f\in \mathcal{O}_X(U)$, $a \in V(U)$ and $d: \mathcal{O}_X \rightarrow \mathcal{O}_X$ is the exterior differential;

\item \textbf{Compatibility}:
\begin{align*}
    \langle \nabla(a),b \rangle+\langle a,\nabla(b) \rangle=\lambda d\langle a,b \rangle,
\end{align*}
where $a,b \in V(U)$.
\end{enumerate}
In this paper, a logarithmic orthogonal/symplectic $\lambda$-connection will be called an \emph{orthogonal/symplectic $\lambda$-connection} for simplicity.
\end{defn}

\begin{rem}
Let $(V,\langle \, , \, \rangle,\nabla)$ be an orthogonal/symplectic sheaf together with an orthogonal/symplectic $\lambda$-connection.		
\begin{enumerate}
\item When $\lambda=1$, we say that $\nabla$ is an \emph{orthogonal/symplectic connection}.
			
\item When $\lambda=0$, $\nabla$ is called an \emph{orthogonal/symplectic Higgs field}. In this case, we prefer to use the notation $\theta := \nabla$, and the pair $(V,\theta)$ is called an \emph{orthogonal/symplectic Higgs bundle}. Here we abuse the terminology by ignoring the integrability condition on $\theta$.
\end{enumerate}
\end{rem}

Next, we introduce Hodge filtrations for $\lambda$-connections in the orthogonal/symplectic case.
\begin{defn}\label{defn_na_Hod_fil}
Given a triple $(V,\langle \, , \, \rangle,\nabla)$ as above, let $\Fil$
\begin{align*}
\Fil: 0=Fil_m \subset Fil_{m-1}  \subset \cdots \subset Fil_0=V
\end{align*}
be an orthogonal/symplectic filtration of $V$. We say that $\Fil $ is an \emph{orthogonal/symplectic Hodge filtration} if $\Fil $ satisfies the \emph{Griffiths transversality condition}:
$$\nabla(Fil_{i+1})\subset Fil_{i}\otimes \Omega_X(\log D), \ i \geq 1.$$
Such a tuple $(V , \langle \, , \, \rangle, \nabla, {\rm Fil})$ is called a \emph{Hodge filtered orthogonal/symplectic $\lambda$-connection}.
\end{defn}

Let $(V, \langle \, , \, \rangle ,\nabla, {\rm Fil})$ be a Hodge filtered orthogonal/symplectic $\lambda$-connection. The Rees construction gives a sheaf
\begin{align*}
    \mathcal{V}= \sum t^{-i} Fil_{i} \otimes \mathcal{O}_{X \times \mathbb{A}^1}
\end{align*}
on $X \times \mathbb{A}^1$, and $t \nabla$ is well-defined on $\mathcal{V}$. Taking the limit $\lim\limits_{t \rightarrow 0} (V,t\nabla)$, we obtain a well-defined Higgs field $\theta$ on ${\rm Gr}_{\rm Fil}(V)$. In conclusion, we obtain an orthogonal/symplectic Higgs sheaf $({\rm Gr}_{\rm Fil}(V), \overline{\langle \, , \rangle}, \theta)$ such that
\begin{align*}
    \theta( \frac{Fil_i}{Fil_{i+1}} ) \subseteq \frac{Fil_{i-1}}{Fil_{i}} \otimes \Omega_X({\rm log} \, D), \ i \geq 1.
\end{align*}
Note that ${\rm Gr}_{\rm Fil}(V)$ is torsion-free because the filtration we choose is saturated. This observation gives the following definition.

\begin{defn}
A \emph{graded orthogonal/symplectic Higgs sheaf} is an orthogonal/symplectic Higgs sheaf $(E, \langle \, , \, \rangle,\theta)$ together with a decomposition  $E=\mathop{\bigoplus}\limits_{i=1}^{m} E_i$ such that
\begin{enumerate}
\item $\theta(E_i)\subset E_{i+1}\otimes \Omega_X(\log D)$ for $ 1\leq i\leq m-1$,
			
\item $\langle \, , \, \rangle|_{E_i\otimes E_j}$ is nondegenerate for $i+j=m+1$ over a nonempty open subset in the locally free locus whose complement is of codimension $\geq 2$,
			
\item $\langle \, , \, \rangle|_{E_i\otimes E_j}=0$ for $i+j\neq m+1$.
\end{enumerate}
\end{defn}

Summing up, as in the case of vector bundles \cite{S10}, one gets a graded orthogonal/symplectic Higgs sheaf from a Hodge filtered orthogonal/symplectic $\lambda$-connection:

\begin{lem}
Let $(V, \langle \, , \, \rangle ,\nabla, {\rm Fil})$ be a Hodge filtered orthogonal/symplectic $\lambda$-connection. It induces a graded orthogonal/symplectic Higgs sheaf $(E,\overline{\langle \, , \, \rangle},\theta)$, where $E:={\rm Gr}_{\rm Fil}(V)$.
\end{lem}

We would like to remind the reader that in the above lemma, the $\lambda$-connection $\nabla$ could be a connection or a Higgs field.

\subsection{Stability Condition and Harder--Narasimhan Filtration}\label{subsect_stab_cond}
We fix an ample line bundle $L$ over $X$. Let $V$ be a torsion free sheaf on $X$. Define $\mu_L(V):=\deg_L(V)/\rk(V)$ to be the slope of $V$ with respect to $L$. In this paper, if there is no ambiguity, we omit the subscript $L$ and use the notation $\mu$ and $\deg$ for slope and degree respectively. In this subsection, a triple $(V,\langle \, , \, \rangle,\nabla)$ is an orthogonal/symplectic sheaf $(V,\langle \, , \, \rangle)$ together with an orthogonal/symplectic $\lambda$-conneciton $\nabla$ unless stated otherwise.
	
\begin{defn}
A triple $(V,\langle \, , \, \rangle,\nabla)$ is \emph{semistable} (resp. \emph{stable}) if for any nontrivial $\nabla$-invariant isotropic subsheaf $W \subset V$, we have $\mu(W)\leqslant 0$ (resp. $<$).
\end{defn}

\begin{rem}
In characteristic zero, this stability condition is equivalent to Ramanathan's stability condition on the corresponding principal bundles \cite{Ra75,Ra96a,Ra96b}. We briefly review the correspondence as follows. A $G$-bundle $E$ is $R$-semistable ($R$-stable) if for any proper parabolic subgroup $P \subseteq G$, any reduction of structure group $\sigma: X \rightarrow E/P$ and any dominant character $\chi: P \rightarrow \mathbb{G}_m$, we have $\deg \chi_* E_\sigma \leq 0$ (resp. $<$), where $E_\sigma := E \times_{\sigma} X$ is a $P$-bundle on $X$ and $\chi_* E_\sigma$ is a line bundle on $X$. When we consider the case of orthogonal/symplectic groups, reductions of structure group correspond exactly to filtrations of isotropic subbundles. Following this idea, the equivalence of the stability conditions for symplectic groups is proven in \cite[Appendix]{KSZ21}, and the same argument holds for orthogonal groups.

In positive characteristic, the $R$-stability condition of $E$ is equivalent to the stability condition of the associated bundle $V$ under a low weight representation $G \hookrightarrow {\rm SL}(W)$ \cite[\S 2]{BP03}. In our case, the representation we take is the standard one, which is obvious of low weight. Therefore, the stability condition considered in this paper is equivalent to the Ramanathan's stability condition for principal bundles in positive characteristic. We give a brief proof in Proposition \ref{prop_HN_fil}.
\end{rem}

\begin{prop}\label{prop_HN_fil}
Given a triple $(V,\langle \, , \, \rangle,\nabla)$, let $$\mathrm{HN}:0=V_0\subset V_1 \subset \cdots \subset V_t=V$$ be the Harder--Narasimhan filtration of $(V, \nabla)$, that is, we forget the orthogonal/symplectic structure of $V$. Then $\mathrm{HN}$ is an orthogonal/symplectic filtration.
\end{prop}
	
\begin{proof}
In fact, $$\mathrm{HN}^\perp:0=V_t^\perp\subset V_{t-1}^\perp  \subset \cdots \subset V_0^\perp=V$$ is again a Harder--Narasimhan filtration of $(V, \nabla)$. By the uniqueness of the Harder--Narasimhan filtration, we have $V_{t-i}=V_{i}^\perp$.
\end{proof}

As a direct consequence, we have the following corollary.
\begin{cor}
Suppose that $(V,\langle \, , \, \rangle,\nabla)$ is unstable. Then the maximal destabilizer of $(V, \nabla)$ is isotropic. Furthermore, $(V,\langle \, , \, \rangle,\nabla)$ is semistable if and only if $(V, \nabla)$ is semistable.
\end{cor}
	
\begin{defn} \label{gr-semistable}
Let $\Fil$ be an orthogonal/symplectic Hodge filtration on $(V , \langle \, ,\, \rangle,\nabla)$. We say that $\Fil$ is \textbf{gr-semistable} if the associated graded orthgonal/symplectic Higgs sheaf $(\Gr_{\Fil}(V), \overline{\langle \, , \, \rangle},\theta)$ is semistable. Moreover, $(V, \langle \, , \, \rangle, \nabla)$ is \emph{gr-semistable} if there exists a gr-semistable filtration $\Fil$ of $(V , \langle \, ,\, \rangle,\nabla)$.
\end{defn}
	
One may ask that whether there always exists a  gr-semistable filtration on a semistable triple $(V , \langle \, ,\, \rangle,\nabla)$, where $\nabla$ is a connection? We will give an answer to this problem in the next sections.

\section{Quasi Gr-semistable Filtration: Odd Rank}\label{sect_odd}

Let $(V, \langle \, , \, \rangle, \nabla)$ be an orthogonal sheaf $(V, \langle \, , \, \rangle)$ together with an orthogonal $\lambda$-connection $\nabla$, where $V$ is a torsion free sheaf of odd rank. In this section, we introduce a special type of orthogonal Hodge filtrations, which is called \emph{quasi gr-semistable filtrations} (Definition \ref{defn_gr_semi_odd}), and prove that the existence of gr-semistable filtrations is equivalent to the existence of quasi gr-semistable filtrations of $(V, \langle \, , \, \rangle, \nabla)$ (Theorem \ref{thm_gr_and_quasi_gr_odd}).

\subsection{Definition and Basic Properties}
	
\begin{defn}[\textbf{Quasi gr-semistable filtration in odd rank case}]\label{defn_gr_semi_odd}
Let $$\Fil: 0= L_0 \subsetneq L_1 \subsetneq \cdots \subsetneq L_m \subsetneq L_m^\perp\subsetneq\cdots \subsetneq L_1^\perp\subsetneq L_0^\perp= V$$ be an orthogonal Hodge filtration of $(V, \langle \, , \, \rangle, \nabla)$. We say that $\Fil$ is \emph{quasi gr-semistable}  if   it satisfies the following two conditions:
		
\begin{enumerate}
\item[ I.] For any saturated Griffiths transverse filtration $$0=W_0\subset W_1 \subset \cdots\subset W_m$$ such that $L_{i-1}\subset W_i\subset L_i$, we have
$$\sum_{0\leq i\leq m}\deg(W_i)\leqslant\sum_{0\leq i\leq m}\deg(L_i).$$
			
\item[II.] For any saturated Griffiths transverse filtration $$0\subset S_1 \subset \cdots\subset S_m\subset S_{m+1}\subset T_m^\perp\subset\cdots \subset T_1^\perp\subset V$$ such that  $$\deg(S_{m+1})>\sum_{0\leq i\leq m}\big(\deg(L_i)-\deg(S_i)\big)+\sum_{0\leq i\leq m}\big(\deg(L_i)-\deg(T_i)\big),$$ we have
$$\nabla(S_{m+1})\subseteq S_{m+1}^\perp \otimes\Omega_X(\log D),$$ where $L_{i-1}\subset S_i \subset T_i \subset L_i,\  L_m\subset S_{m+1}\subset L_m^\perp$ and $S_{m+1}$ is isotropic.
\end{enumerate}
Moreover, we say that $(V,\langle \, , \, \rangle,\nabla)$ is \emph{quasi gr-semistable} (or $\nabla$ is a \emph{quasi gr-semistable connection}) if there exists a quasi gr-semistable filtration on $V$.
\end{defn}

\begin{examp}\label{smallm}
In this example, we briefly explain the quasi gr-semistability condition for small $m$.

\begin{itemize}
\item  $m=0$.
			
The trivial filtration $$\Fil:0\subset V$$ is  quasi gr-semistable if and only if $\nabla(W)\subset W^\perp\otimes\Omega_X(\log D)$ holds for any isotropic saturated subsheaf $W$ of $\deg(W)>0$.
			
\item  $m=1$.
			
An orthogonal Hodge filtration $0\subset N\subset N^\perp\subset V$ is  quasi gr-semistable if and only if it satisfies the following two conditions:
\begin{enumerate}
\item $\mu_{\min}(N)\geq 0$;
\item for any saturated Griffiths transverse flitration $0\subset S_1\subset S_2\subset T_1^\perp$ such that
\begin{itemize}
    \item[$\bullet$] $S_1\subset T_1\subset N,\ N\subset S_2\subset N^\perp $,
    \item[$\bullet$] $S_2$ is isotropic,
    \item[$\bullet$] $\deg(S_2)>\deg(N)-\deg(S_1)+\deg(N)-\deg(T_1)$,
\end{itemize}
we have $\nabla(S_2)\subset S_2^\perp\otimes \Omega_X(\log D).$
\end{enumerate}
\end{itemize}
When $m=0$, quasi gr-semistability is easy to check. However, when $m \geq 1$, the stability condition becomes much more complicated. The above discussion will be used in \S\ref{sect_small_rank}.
\end{examp}

\begin{rem}\label{explainregular}
In this remark, we briefly explain Condition I. and II. in terms of graded Higgs sheaves. Given $(V, \langle \, , \, \rangle, \nabla, \Fil)$, let $(\Gr_{\Fil}, \overline{\langle \, , \, \rangle}, \theta)$ be the corresponding graded orthogonal Higgs sheaf, where
\begin{align*}
    \Gr_{\Fil}(V)=\mathop{\bigoplus}\limits_{i=1}^{m} \frac{L_i}{L_{i-1}}\oplus \frac{L_m^\perp}{L_m}\oplus \mathop{\bigoplus}\limits_{i=1}^{m} \frac{L_{m-i}^\perp}{L_{m-i+1}^\perp}.
\end{align*}

\begin{enumerate}
\item The filtration $$0=W_0\subset W_1 \subset \cdots\subset W_m$$ in Condition I. gives a Higgs subsheaf of $\Gr_{\Fil}(V)$ $$\alpha(W.):=\mathop{\bigoplus}\limits_{i=1}^{m} \frac{W_i}{L_{i-1}}\oplus \frac{L_m^\perp}{L_m}\oplus \mathop{\bigoplus}\limits_{i=1}^{m} \frac{L_{m-i}^\perp}{L_{m-i+1}^\perp}.$$
Then the inequality $$\sum_{0\leq i\leq m}\deg(W_i)\leqslant\sum_{0\leq i\leq m}\deg(L_i).$$ is equivalent to $\deg\big(\alpha(W.)\big)\leq 0.$		
\item The filtration $$0\subset S_1 \subset \cdots\subset S_m\subset S_{m+1}\subset T_m^\perp\subset\cdots \subset T_1^\perp\subset V$$ in Condition II. defines an isotropic Higgs subsheaf of $\Gr_{\Fil}(V)$
$$\beta(S.,T.):=\mathop{\bigoplus}\limits_{i=1}^{m} \frac{S_i}{L_{i-1}}\oplus \frac{S_{m+1}}{L_m}\oplus \mathop{\bigoplus}\limits_{i=1}^{m} \frac{T_{m-i+1}^\perp}{L_{m-i+1}^\perp}.$$
Then the inequality $$\deg(S_{m+1})>\sum_{0\leq i\leq m}\big(\deg(L_i)-\deg(S_i)\big)+\sum_{0\leq i\leq m}\big(\deg(L_i)-\deg(T_i)\big)$$ is equivalent to $\deg\big(\beta(S.,T.)\big)>0$.
\end{enumerate}
Summing up, $\Fil$ is quasi gr-semistable if and only if the following two conditions hold:
\begin{enumerate}
\item $\deg\big(\alpha(W.)\big)\leq 0$;
			
\item if $\deg\big(\beta(S.,T.)\big)>0$, then we have $\nabla(S_{m+1})\subseteq S_{m+1}^\perp \otimes\Omega_X(\log D)$.
\end{enumerate}
\end{rem}
	
\begin{prop}\label{ssregular}
If an orthogonal Hodge filtration $\Fil $ is gr-semistable of $(V,\langle \, , \, \rangle,\nabla)$, then $\Fil$ is quasi gr-semistable.
\end{prop}

\begin{proof}
The semistability of $\big(\Gr_{\Fil}(V), \overline{\langle \, , \, \rangle}, \theta\big)$ implies that $\deg(W)\leq 0$ holds for any isotropic Higgs subsheaf of $\big(\Gr_{\Fil}(V),\theta\big)$. By Prop \ref{prop_HN_fil}, $\deg(W)\leq 0$ holds for any Higgs subsheaf of $\big(\Gr_{\Fil}(V),\theta\big)$. Hence for any saturated Griffiths transverse filtration given in Condition I. and II., we have $\deg\big(\alpha(W.)\big)\leq 0$ and $\deg\big(\beta(S.,T.)\big)\leq0$. Then the first condition in Remark \ref{explainregular} is satisfied automatically, and the second condition also holds because there is no $\beta(S.,T.)$ such that $\deg( \beta(S.,T.)) > 0$.
\end{proof}
	
\begin{prop}\label{positive}
Let  $$\Fil: 0= L_0 \subsetneq L_1 \subsetneq \cdots \subsetneq L_m \subsetneq L_m^\perp\subsetneq\cdots  \subsetneq L_1^\perp\subsetneq L_0^\perp= V$$ be a nontrivial quasi gr-semistable filtration of $(V, \langle \, , \, \rangle, \nabla)$. Then we have $\deg(L_i)\geq 0$. Moreover  $\deg(F)\leq \deg(L_1)$ holds for any subsheaf $F\subset L_1.$
\end{prop}
	
\begin{proof}
For a fixed $1\leq i_0\leq m$, consider the following Griffiths transverse filtration
$$0=W_0\subset W_1\subset \cdots\subset W_m$$
where $W_i= L_{i-1}$ for $1\leq i\leq i_0$ and $W_i= L_{i+1}$ for $i_0+1\leq i\leq m$. Then the inequality  $$\sum_{0\leq i\leq m}\deg(W_i)\leqslant\sum_{0\leq i\leq m}\deg(L_i)$$ implies that $\deg(L_{i_0})\geq 0.$ Similarly, let $W_1=F$ and $W_i=L_i$ for $2\leq i\leq m$, and we get $\deg(F)\leq \deg(L_1)$.
\end{proof}	

\subsection{The Existence of Gr-semistable  Filtration}
	
In this subsection, we will show that given a semistable triple $(V, \langle \, , \, \rangle, \nabla)$, the quasi gr-semistability of $(V, \langle \, , \, \rangle, \nabla)$ implies the existence of gr-semistable filtrations (Theorem \ref{main thm}). Combing with Proposition \ref{ssregular}, we obtain the main result (Theorem \ref{thm_gr_and_quasi_gr_odd}) in this section. First, let us begin with a lemma about the maximal destabilizer of a graded orthogonal Higgs sheaf.
	
\begin{lem}\cite[Lemma A.6]{LSZ19}\label{destabilizer}
Let $(E=\mathop{\bigoplus}\limits_{i=1}^{m} E_i,\ \langle \, , \, \rangle,\ \theta)$ be an unstable graded orthogonal Higgs sheaf. Let $M\subset E$ be the maximal destabilizer of $(E,\theta)$. Then $M$ is an isotropic saturated graded Higgs subsheaf, that is , $M=\mathop{\bigoplus}\limits_{i=1}^{m} M_i$ with $M_i=M\cap E_i$ and the Higgs field is induced by $\theta$.
\end{lem}

Now let $\Fil$
$$\Fil: 0= L_0 \subsetneq L_1 \subsetneq \cdots \subsetneq L_m \subsetneq L_m^\perp\subsetneq\cdots \subsetneq L_1^\perp\subsetneq L_0^\perp= V$$
be a quasi gr-semistable filtration of $(V,\langle \, , \, \rangle,\nabla)$, and we consider the associated orthogonal graded Higgs sheaf $(\Gr_{\Fil}(V), \overline{\langle \, , \, \rangle},\theta)$, where $$\Gr_{\Fil}(V)=\mathop{\bigoplus}\limits_{i=1}^{m} \frac{L_i}{L_{i-1}}\oplus \frac{L_m^\perp}{L_m}\oplus \mathop{\bigoplus}\limits_{i=1}^{m} \frac{L_{m-i}^\perp}{L_{m-i+1}^\perp}.$$ If $(\Gr_{\Fil}(V),\theta)$ is not semistable, then we have a maximal destabilizer and denote it by $M_{\Fil}$. By Lemma \ref{destabilizer}, we have
$$M_{\Fil}=\mathop{\bigoplus}\limits_{i=1}^{m} \frac{M_i}{L_{i-1}}\oplus \frac{M_{m+1}}{L_m}\oplus \mathop{\bigoplus}\limits_{i=1}^{m} \frac{N_{m-i+1}^\perp}{L_{m-i+1}^\perp},$$ where $L_{i-1}\subset M_i\subset N_i\subset L_i$ for $1\leq i\leq m$ and $M_{m+1}$ is an isotropic subsheaf. We construct a new filtration $\xi(\Fil)$ from $\Fil$ as follows:
$$\xi(\Fil):0=M_0\subset M_1\subset\cdots\subset M_{m+1}\subset  M_{m+1}^\perp\subset\cdots\subset M_1^\perp\subset M_0^\perp=V.$$
Moreover, $\xi$ can be regarded as an operator on the set of Griffiths transverse filtrations.
	
\begin{lem}\label{preserve}
Let $\Fil$ be a quasi gr-semistable filtration of $(V, \langle \, , \, \rangle, \nabla)$. Then $\xi(\Fil)$ is a (orthogonal) Hodge filtration. Moreover, $\xi(\Fil)$ is also quasi gr-semistable.
\end{lem}
	
\begin{proof}
First, we will show that $\xi(\Fil)$ is a Hodge filtration, that is, $\xi(\Fil)$ satisfies the Griffiths transversality condition. It is enough to show that $$\nabla(M_{m+1})\subset M_{m+1}^\perp\otimes\Omega_X(\log D).$$ This is true by Remark \ref{explainregular} since $\Fil$ is quasi gr-semistable and $\deg(M_{\Fil})>0$.

Next, we will check the quasi gr-semistability of $\xi(\Fil)$.
\begin{enumerate}
\item[I.] Let $$0=W_0\subset W_1 \subset \cdots\subset W_{m+1}$$  be a given saturated Griffiths transverse filtration such that $M_{i-1}\subset W_i\subset M_i$. We need to show that $$\sum_{0\leq i\leq m+1}\deg(W_i)\leqslant\sum_{0\leq i\leq m+1}\deg(M_i).$$
Consider the following Higgs subsheaf $\sigma(W.)$ of $\Gr_{\Fil}(V)$:
$$\sigma(W.):=\mathop{\bigoplus}\limits_{i=1}^{m} \frac{W_i+L_{i-1}}{L_{i-1}}\oplus\frac{W_{m+1}+L_m}{L_m} \oplus \mathop{\bigoplus}\limits_{i=1}^{m} \frac{N_{m-i+1}^\perp}{L_{m-i+1}^\perp}.$$
Clearly, $\sigma(W.)$ is a Higgs subsheaf of $M_{\Fil}$. Since $M_{\Fil}$ is the maximal destabilizer, we have
\begin{equation}
\deg(M_{\Fil})-\deg(\sigma(W.))\geq 0
\end{equation}
On the other hand,
\begin{equation*}
\begin{split}
&\deg(M_{\Fil})-\deg(\sigma(W.))=\sum_{0\leq i\leq m+1}\deg(M_i)-\sum_{0\leq i\leq m+1}\deg(W_i+L_{i-1})\\
&=\sum_{0\leq i\leq m+1}\deg(M_i)-\sum_{0\leq i\leq m+1}\big(\deg(W_i)+\deg(L_{i-1})-\deg(W_i\cap L_{i-1})\big)\\
&=\sum_{0\leq i\leq m+1}\big(\deg(M_i)-\deg(W_i)\big)-\sum_{0\leq i\leq m+1}\big(\deg(L_{i-1})-\deg(W_i\cap L_{i-1})\big).
\end{split}
\end{equation*}
Therefore,
\begin{equation}\label{increasing}
\sum_{0\leq i\leq m+1}\big(\deg(M_i)-\deg(W_i)\big)\geq \sum_{0\leq i\leq m+1}\big(\deg(L_{i-1})-\deg(W_i\cap L_{i-1})\big)\geq 0.
\end{equation}
The last inequality holds because $\Fil$ is quasi gr-semistable. Condition I. in Definition \ref{defn_gr_semi_odd} is satisfied.

\hfill{\space}

\item[II.]
Let
$$\Fil_{ST}:0\subset S_1 \subset \cdots\subset S_{m+1}\subset S_{m+2}\subset T_{m+1}^\perp\subset\cdots \subset T_1^\perp\subset V$$ be a given saturated Griffiths transverse flitration such that
\begin{itemize}
    \item $M_{i-1}\subset S_i \subset T_i \subset M_i,\  M_{m+1}\subset S_{m+2}\subset M_{m+1}^\perp$,
    \item $S_{m+2}$ is isotropic,
    \item the inequality
    \begin{equation}\label{inequality}
	\deg(S_{m+2})>\sum_{0\leq i\leq m+1}\big(\deg(M_i)-\deg(S_i)\big)+\sum_{0\leq i\leq m+1}\big(\deg(M_i)-\deg(T_i)\big).
    \end{equation}
    is satisfied.
\end{itemize}
We need to show that $\nabla(S_{m+2})\subset S_{m+2}^\perp\otimes \Omega_X(\log D).$ Consider the following filtration:
$$\Fil_{\tilde{S}\tilde{T}}:=0\subset \tilde{S}_1\subset \cdots\subset \tilde{S}_m\subset \tilde{S}_{m+1}\subset \tilde{T}_m^\perp\subset\cdots \subset \tilde{T}_1^\perp\subset V,$$
where $\tilde{S}_i=S_{i+1}\cap L_i,\ \tilde{T}_i=T_{i+1}\cap L_i$, for $1\leq i\leq m,$ and $\tilde{S}_{m+1}=S_{m+2}$. Then $L_{i-1}\subset \tilde{S}_i\subset \tilde{T}_i\subset L_i$ for $1\leq i\leq m$, $L_m\subset\tilde{S}_{m+1}\subset L_m^\perp$ and $\tilde{S}_{m+1}$ is isotropic. Moreover, we have $$\deg(\tilde{S}_{m+1})>\sum_{0\leq i\leq m}\big(\deg(L_i)-\deg(\tilde{S}_i)\big)+\sum_{0\leq i\leq m}\big(\deg(L_i)-\deg(\tilde{T}_i)\big).$$
In fact, by the inequality   \ref{increasing} and \ref{inequality}, we have
\begin{equation}
\begin{split}
    \deg(\tilde{S}_{m+1})&=\deg(S_{m+2})\\
    &>\sum_{0\leq i\leq m+1}\big(\deg(M_i)-\deg(S_i)\big)+\sum_{0\leq i\leq m+1}\big(\deg(M_i)-\deg(T_i)\big)\\
    &\geq \sum_{0\leq i\leq m+1}\big(\deg(L_{i-1})-\deg(S_i\cap L_{i-1})\big) \\
    & +\sum_{0\leq i\leq m+1}\big(\deg(L_{i-1})-\deg(T_i\cap L_{i-1})\big)\\
    &=\sum_{0\leq i\leq m}\big(\deg(L_i)-\deg(\tilde{S}_i)\big)+\sum_{0\leq i\leq m}\big(\deg(L_i)-\deg(\tilde{T}_i)\big).
\end{split}
\end{equation}
Then, by the quasi gr-semistability of $\Fil$, we have $$\nabla(\tilde{S}_{m+1})\subset\tilde{S}_{m+1}^\perp\otimes\Omega_X(\log D),$$ which implies $$\nabla(S_{m+2})\subset S_{m+2}^\perp\otimes \Omega_X(\log D).$$
\end{enumerate}
\end{proof}

Now we will state and prove the main result in this section.
\begin{thm}\label{main thm}
Let $(V,\langle \, , \, \rangle,\nabla)$ be semistable and quasi gr-semistable with odd rank. There exists a gr-semistable filtration on $(V,\langle \, , \, \rangle,\nabla)$.
\end{thm}

We need the following two lemmas to prove this theorem.
\begin{lem}\label{decrease}
Let $\Fil$ be a quasi gr-semistable filtration of $(V, \langle \, , \, \rangle, \nabla)$. Asssume that the graded orthogonal sheaves $(\Gr_{\Fil}(V),\theta_1)$ and $(\Gr_{\xi(\Fil)}(V),\theta_2)$ are not semistable. Let $M_{\Fil}$ (resp. $M_{\xi(\Fil)}$)  be the maximal destabilizer of $(\Gr_{\Fil}(V),\theta_1)$ (resp. $\Gr_{\xi(\Fil)}(V),\theta_2)$). Then we have a natural morphism of graded Higgs sheaves
$$\Phi: M_{\xi(\Fil)}\longrightarrow M_{\Fil}$$
such that
\begin{enumerate}
\item $\mu(M_{\xi(\Fil)})\leq \mu(M_{\Fil})$;
\item if $\mu(M_{\xi(\Fil)})= \mu(M_{\Fil})$, the morphism $\Phi$ is injective, then $\rk(M_{\xi(\Fil)})\leq\rk(M_{\Fil})$;
\item if $\big(\mu(M_{\xi(\Fil)},\ \rk(M_{\xi(\Fil)})\big)= \big(\mu(M_{\Fil}),\rk(M_{\Fil})\big)$, then $\Phi$ is an isomorphism on an open subset $U \subset X$ with ${\rm codim}(X/U) \geq 2$.
\end{enumerate}
\end{lem}

\begin{proof}
We write $\Fil$ as
$$\Fil: 0= L_0 \subsetneq L_1 \subsetneq \cdots \subsetneq L_m \subsetneq L_m^\perp\subsetneq\cdots \subsetneq L_1^\perp\subsetneq L_0^\perp= V$$
Then,
$$\Gr_{\Fil}(V)=\mathop{\bigoplus}\limits_{i=1}^{m} \frac{L_i}{L_{i-1}}\oplus \frac{L_m^\perp}{L_m}\oplus \mathop{\bigoplus}\limits_{i=1}^{m} \frac{L_{m-i}^\perp}{L_{m-i+1}^\perp}.$$
Then, let $M_{\Fil}$ be as follows
$$M_{\Fil}=\mathop{\bigoplus}\limits_{i=1}^{m} \frac{M_i}{L_{i-1}}\oplus \frac{M_{m+1}}{L_m}\oplus \mathop{\bigoplus}\limits_{i=1}^{m} \frac{N_{m-i+1}^\perp}{L_{m-i+1}^\perp}.$$
Recall that
\begin{itemize}
\item $L_{i-1} \subset M_i \subset N_i \subset L_i$ for $1 \leq i \leq m$,
\item $M_{m+1}$ is isotropic.
\end{itemize}
Then, $\xi(\Fil)$ is given by
$$\xi(\Fil):0=M_0\subset M_1\subset\cdots\subset M_{m+1}\subset  M_{m+1}^\perp\subset\cdots\subset M_1^\perp\subset M_0^\perp=V,$$
and similarly,
$$\Gr_{\xi(\Fil)}(V)=\mathop{\bigoplus}\limits_{i=1}^{m+1} \frac{M_i}{M_{i-1}}\oplus \frac{M_{m+1}^\perp}{M_{m+1}}\oplus \mathop{\bigoplus}\limits_{i=1}^{m+1} \frac{M_{m+1-i}^\perp}{M_{m+2-i}^\perp},$$
such that
\begin{itemize}
\item $M_{i-1} \subset X_i \subset Y_i \subset M_i$ for $1 \leq i \leq m+1$,
\item $X_{m+2}$ is isotropic.
\end{itemize}
Therefore, $M_{\xi(\Fil)}$ is given as follows:
$$M_{\xi(\Fil)}=\mathop{\bigoplus}\limits_{i=1}^{m+1} \frac{X_i}{M_{i-1}}\oplus \frac{X_{m+2}}{M_{m+1}}\oplus \mathop{\bigoplus}\limits_{i=1}^{m+1} \frac{Y_{m+2-i}^\perp}{M_{m+2-i}^\perp}.$$
		
Now we define the following natural morphisms of sheaves:
\begin{equation}
\begin{split}
&\Phi_i:\frac{X_i}{M_{i-1}}\longrightarrow \frac{M_i}{L_{i-1}},\ 1\leq i\leq m+1,\\
&\Phi_{m+2}:\frac{X_{m+2}}{M_{m+1}}\longrightarrow 0,\\&\tilde{\Phi}_j:\frac{Y_{m+2-j}^\perp}{M_{m+2-j}^\perp}\longrightarrow 0,\ 1\leq j\leq m+1,\\
\end{split}
\end{equation}
and we obtain a morphism of graded Higgs sheaves
$$\Phi: M_{\xi(\Fil)}\longrightarrow M_{\Fil}.$$
\begin{enumerate}
\item If $\Phi\neq 0$, we have $$\mu(M_{\xi(\Fil)})\leq \mu(M_{\Fil}).$$

\item If $\Phi= 0$, then $X_i\subset L_{i-1},\ 1\leq i\leq m+1$. Also, we have $X_{m+2}\subset L_{m}^\perp$ and $\ Y_{m+2-j}^\perp \subset L_{m-j}^\perp$ for $1\leq j\leq m+1$. In this case, we consider the quotient Higgs sheaf
$$\frac{\Gr_{\Fil}(V)}{M_{\Fil}}=\mathop{\bigoplus}\limits_{i=1}^{m} \frac{L_i}{M_i}\oplus \frac{L_m^\perp}{M_{m+1}}\oplus \mathop{\bigoplus}\limits_{i=1}^{m} \frac{L_{m-i}^\perp}{N_{m-i+1}^\perp},$$
and define the following natural morphisms:
\begin{equation}
\begin{split}
&\Psi_i:\frac{X_i}{M_{i-1}}\longrightarrow \frac{L_{i-1}}{M_{i-1}},\ 2\leq i\leq m+1,\\
&\Psi_{m+2}:\frac{X_{m+2}}{M_{m+1}}\longrightarrow \frac{L_m^\perp}{M_{m+1}},\\
&\tilde{\Psi}_j:\frac{Y_{m+2-j}^\perp}{M_{m+2-j}^\perp}\longrightarrow \frac{L_{m-j}^\perp}{N_{m-j+1}^\perp},\ 1\leq j\leq m,\\
&\tilde{\Psi}_{m+1}:\frac{Y_1^\perp}{M_1^\perp}\longrightarrow 0 .
\end{split}
\end{equation}
This construction gives a morphism of graded Higgs sheaves
$$\Psi: M_{\xi(\Fil)}\longrightarrow \frac{\Gr_{\Fil}(V)}{M_{\Fil}}.$$
We \textbf{claim} that $\Psi\neq 0$. Then we have $$\mu(M_{\xi(\Fil)})\leq \mu(\mathrm{Im}(\Psi)) <\mu(M_{\Fil}).$$
Note that in this case we get a strict inequality
$$\mu(M_{\xi(\Fil)})<\mu(M_{\Fil}).$$
The first statement that $\mu(M_{\xi(\Fil)})\leq \mu(M_{\Fil})$ is proven.

\hfill{\space}

For the \textbf{claim}, we suppose that $\Psi=0$. Then $\frac{X_i}{M_{i-1}}=0$ and thus $X_i  = M_{i-1}$ for $1 \leq i \leq m+2$. We have
$$M_{\xi(\Fil)}=\mathop{\bigoplus}\limits_{i=1}^{m+1} \frac{Y_{m+2-i}^\perp}{M_{m+2-i}^\perp}.$$
Since $(\Gr_{\xi(\Fil)}(V), \theta_2)$ is not semistable and $M_{\xi(\Fil)}$ is the maximal destabilizer, we have $$\deg(M_{\xi(\Fil)})>0.$$
However, Lemme \ref{preserve} tells us that $\xi(\Fil)$ is quasi gr-semistable, and then
\begin{equation*}
\deg(M_{\xi(\Fil)})=\sum_{0\leq i\leq m+1}\big(\deg(Y_i)-\deg(M_i)\big)\leq 0
\end{equation*}
by Condition I. in Definition \ref{defn_gr_semi_odd}. This is a contradiction.
\end{enumerate}

\hfill{\space}

The above discussion shows that $\mu(M_{\xi(\Fil)})= \mu(M_{\Fil})$ only if when $\Phi$ is injective. Then the second statement also holds. Moreover, if $\big(\mu(M_{\xi(\Fil)},\ \rk(M_{\xi(\Fil)})\big)= \big(\mu(M_{\Fil}),\rk(M_{\Fil})\big)$, clearly $\Phi$ is an isomorphism on an open subset $U \subset X$ with ${\rm codim}(X/U) \geq 2$, which gives the third statement.
\end{proof}
	
\begin{lem}\label{lem_const}
Let $(V, \langle \, , \, \rangle, \nabla)$ be semistable and quasi gr-semistable. If the sequence of pairs $\{\big(\mu(M_{\xi^k(\Fil)}),\ \rk(M_{\xi^k(\Fil)})\big)\}_{k\geq 0}$ is constant, then the graded orthogonal Higgs sheaf $(\Gr_{\Fil}(V), \overline{\langle \, , \, \rangle}, \theta)$ is semistable.
\end{lem}
		
\begin{proof}
By Proposition \ref{prop_HN_fil}, it is equivalent to work on $(\Gr_{\Fil}(V), \theta)$ and we assume that $(\Gr_{\Fil}(V), \theta)$ is not semistable. We use the same notations as in the proof of Lemma \ref{decrease}:
$$\Fil: 0= L_0 \subsetneq L_1 \subsetneq \cdots \subsetneq L_m \subsetneq L_m^\perp\subsetneq\cdots \subsetneq L_1^\perp\subsetneq L_0^\perp= V,$$
			
$$\Gr_{\Fil}(V)=\mathop{\bigoplus}\limits_{i=1}^{m} \frac{L_i}{L_{i-1}}\oplus \frac{L_m^\perp}{L_m}\oplus \mathop{\bigoplus}\limits_{i=1}^{m} \frac{L_{m-i}^\perp}{L_{m-i+1}^\perp},$$
			
$$M_{\Fil}=\mathop{\bigoplus}\limits_{i=1}^{m} \frac{M_i}{L_{i-1}}\oplus \frac{M_{m+1}}{L_m}\oplus \mathop{\bigoplus}\limits_{i=1}^{m} \frac{N_{m-i+1}^\perp}{L_{m-i+1}^\perp},$$ 			

$$\xi(\Fil):0=M_0\subset M_1\subset\cdots\subset M_{m+1}\subset  M_{m+1}^\perp\subset\cdots\subset M_1^\perp\subset M_0^\perp=V,$$			

$$\Gr_{\xi(\Fil)}(V)=\mathop{\bigoplus}\limits_{i=1}^{m+1} \frac{M_i}{M_{i-1}}\oplus \frac{M_{m+1}^\perp}{M_{m+1}}\oplus \mathop{\bigoplus}\limits_{i=1}^{m+1} \frac{M_{m+1-i}^\perp}{M_{m+2-i}^\perp},$$								$$M_{\xi(\Fil)}=\mathop{\bigoplus}\limits_{i=1}^{m+1} \frac{X_i}{M_{i-1}}\oplus \frac{X_{m+2}}{M_{m+1}}\oplus \mathop{\bigoplus}\limits_{i=1}^{m+1} \frac{Y_{m+2-i}^\perp}{M_{m+2-i}^\perp}.$$
By Lemma \ref{decrease}, the morphism $\Phi:M_{\xi(\Fil)}\longrightarrow M_{\Fil}$ is an isomorphism over an open subset $U$, where ${\rm codim}(X/U) \geq 2$. Now we always restrict to $U$, and say $\Phi$ is an isomorphism for convenience. Now we suppose that $\xi^{m+2}(\Fil)$ is of the form
$$\xi^{m+2}(\Fil): 0= Z_0 \subset Z_1 \subset \cdots \subset Z_{2m+2}\subset Z_{2m+2}^\perp\subset\cdots \subset Z_1^\perp\subset Z_0^\perp= V.$$
The isomorphisms
\begin{align*}
    M_{\xi^{m+2}(\Fil)} \rightarrow M_{\xi^{m+1}(\Fil)} \rightarrow \dots \rightarrow M_{\Fil}
\end{align*}
imply
$$M_{\xi^{m+2}(\Fil)}=Z_1\oplus \frac{Z_2}{Z_1}\oplus\cdots\oplus \frac{Z_{2m+1}}{Z_{2m}},$$
ant thus $Z_{2m+1}$ is a $\nabla$-invariant subsheaf of $V$. Since $M_{\xi^{m+2}(\Fil)}$ is the maximal destabilizer, we have $$\mu(Z_{2m+1})=\mu(M_{\xi^{m+2}(\Fil)})>0.$$
This contradicts with the condition that $(V,\langle \, , \, \rangle,\nabla)$ is semistable.
\end{proof}

\begin{proof}[\textbf{Proof of Theorem \ref{main thm}}]
First, we take an arbitrary quasi gr-semistable  filtration $\Fil$ of $(V,\langle \, , \, \rangle,\nabla)$. The sequence $\{\big(\mu(M_{\xi^k(\Fil)}),\ \rk(M_{\xi^k(\Fil)})\big)\}_{k\geq 0}$ decreases in the lexicographic ordering as $k$ grows by Lemma \ref{decrease}. Thus, there exists a positive integer $k_0$ such that the sequence $\{\big(\mu(M_{\xi^k(\Fil)}),\ \rk(M_{\xi^k(\Fil)})\big)\}_{k\geq k_0}$ becomes constant. By Lemma \ref{lem_const}, we finish the proof of this theorem.
\end{proof}

\begin{rem}
The statement of Lemma \ref{decrease} also works for arbitrary $(V, \langle \, , \, \rangle, \nabla)$, which could be unstable, while Lemma \ref{lem_const} only holds for semistable triples $(V, \langle \, , \, \rangle, \nabla)$.
\end{rem}

Combining Proposition \ref{ssregular} with Theorem \ref{main thm}, we have
	
\begin{thm}\label{thm_gr_and_quasi_gr_odd}
Let $(V,\langle \, , \, \rangle,\nabla)$ be a semistable orthogonal sheaf together with an orthogonal connection $\nabla$ such that the rank of $V$ is odd. Then, $(V,\langle \, , \, \rangle,\nabla)$ is gr-semistable  if and only if $(V,\langle \, , \, \rangle,\nabla)$ is quasi gr-semistable.
\end{thm}
	
\section{Quasi Gr-semistable Filtration: Even Rank}\label{sect_even}
In this section, we discuss the quasi gr-semistable (orthogonal/symplectic) filtrations of $(V, \langle \, , \, \rangle, \nabla)$, where $\rk(V)$ is even. In the even rank case, there are two types of orthogonal/symplectic Hodge filtrations, which depend on the parity of lengths of the filtrations. Despite this fact, the proof of the main result (Theorem \ref{thm_even}) in the even rank case is almost the same as that of Theorem \ref{main thm} in the odd rank case. Therefore, we only give constructions and state results rather than giving a detailed proof.

\begin{defn}
An orthogonal/symplectic Hodge filtration
$$\Fil: 0=Fil_0\subsetneq Fil_1 \subsetneq Fil_2 \subsetneq \cdots\subsetneq Fil_t=V$$
of $(V, \langle \, , \, \rangle, \nabla)$ is an \emph{odd} (resp. \emph{even}) type filtration if its length $t$ is odd (resp. even).
\end{defn}

\begin{rem}\
\begin{enumerate}
\item An odd type orthogonal Hodge filtration $\Fil$ can be written as
$$\Fil: 0= L_0 \subsetneq L_1 \subsetneq \cdots \subsetneq L_m \subsetneq L_m^\perp\subsetneq\cdots  \subsetneq L_1^\perp\subsetneq L_0^\perp= V,$$
where $L_m$ is not Lagrangian (i.e. $L_m\neq L_m^\perp$) and satisfies $\nabla(L_m)\subseteq L_m^\perp \otimes\Omega_X(\log D).$

\item An even type orthogonal Hodge filtration $\Fil$ can be regarded as
$$\Fil: 0= L_0 \subsetneq L_1 \subsetneq \cdots\subsetneq L_{m-1} \subsetneq L_m \subsetneq L_{m-1}^\perp\subsetneq\cdots  \subsetneq L_1^\perp\subsetneq L_0^\perp= V,$$
where $L_m$ is Lagrangian (i.e. $L_m= L_m^\perp$) subsheaf and satisfies $\nabla(L_m)\subseteq L_{m-1}^\perp \otimes\Omega_X(\log D).$
\end{enumerate}
In summary, the odd type orthogonal Hodge filtrations with even rank are similar to  orthogonal Hodge filtrations with odd rank.
\end{rem}
		
\begin{defn}[\textbf{Quasi gr-semistable filtration in even rank case}]\label{defn_gr_semi_even}
Given $(V,\langle \, , \, \rangle,\nabla)$, let $\Fil$ be an orthogonal/symplectic Hodge filtration.		
\begin{enumerate}
\item If $$\Fil: 0= L_0 \subsetneq L_1 \subsetneq \cdots \subsetneq L_m \subsetneq L_m^\perp\subsetneq\cdots  \subsetneq L_1^\perp\subsetneq L_0^\perp= V$$
is an odd type filtration, $\Fil$ is \emph{quasi gr-semistable} if the following two conditions hold:		
\begin{enumerate}
\item[ I.] For any saturated Griffiths transverse filtration $$0=W_0\subset W_1 \subset \cdots\subset W_m$$ such that $L_{i-1}\subset W_i\subset L_i$, we have
$$\sum_{0\leq i\leq m}\deg(W_i)\leqslant\sum_{0\leq i\leq m}\deg(L_i).$$

\item[ II.] For any saturated Griffiths transverse flitration $$0\subset S_1 \subset \cdots\subset S_m\subset S_{m+1}\subset T_m^\perp\subset\cdots \subset T_1^\perp\subset V$$ such that  $$\deg(S_{m+1})>\sum_{0\leq i\leq m}\big(\deg(L_i)-\deg(S_i)\big)+\sum_{0\leq i\leq m}\big(\deg(L_i)-\deg(T_i)\big),$$
we have
$$\nabla(S_{m+1})\subseteq S_{m+1}^\perp \otimes\Omega_X(\log D),$$
where $L_{i-1}\subset S_i \subset T_i \subset L_i,\  L_m\subset S_{m+1}\subset L_m^\perp$ and $S_{m+1}$
is isotropic.\\
\end{enumerate}
			
\item If $$\Fil: 0= L_0 \subsetneq L_1 \subsetneq \cdots\subsetneq L_{m-1} \subsetneq L_m \subsetneq L_{m-1}^\perp\subsetneq\cdots  \subsetneq L_1^\perp\subsetneq L_0^\perp= V$$ is an even type filtration, $\Fil$ is \emph{quasi gr-semistable} if the following two conditions hold:
\begin{enumerate}
\item[ I.] For any saturated Griffiths transverse filtration $$0=W_0\subset W_1 \subset \cdots\subset W_m$$ such that $L_{i-1}\subset W_i\subset L_i$, we have
$$\sum_{0\leq i\leq m}\deg(W_i)\leqslant\sum_{0\leq i\leq m}\deg(L_i).$$

\item[ II.] For any saturated Griffiths transverse flitration $$0\subset S_1 \subset \cdots\subset S_m\subset T_m^\perp\subset\cdots \subset T_1^\perp\subset V$$ such that  $$\deg(S_m)>\sum_{0\leq i\leq m-1}\big(\deg(L_i)-\deg(S_i)\big)+\sum_{0\leq i\leq m}\big(\deg(L_i)-\deg(T_i)\big),$$ we have
$$\nabla(S_m)\subseteq L_m \otimes\Omega_X(\log D),$$
where $L_{i-1}\subset S_i \subset T_i \subset L_i$ for $1\leq i\leq m$.
\end{enumerate}
\end{enumerate}
Furthermore, we say that $(V,\langle \, , \, \rangle,\nabla)$ is \emph{quasi gr-semistable} (or $\nabla$ is a \emph{quasi gr-semistable $\lambda$-connection}) if there exists a quasi gr-semistable filtration on $V$.
\end{defn}

The following two propositions can be proved in the same manner as Propositions \ref{ssregular} and \ref{positive} in the odd rank case.
	
\begin{prop}\label{prop_ssregular_even}
If an orthogonal/symplectic filtration ${\rm Fil}$ of $(V,\langle \, , \, \rangle,\nabla)$ is gr-semistable, then it is also quasi gr-semistable.
\end{prop}
	
\begin{prop}\label{positiveeven}
Let  $$\Fil: 0= L_0 \subsetneq L_1 \subsetneq \cdots \subsetneq L_t = V$$
be a nontrivial quasi gr-semistable filtration  (either of odd or even type) of $(V,\langle \, , \, \rangle,\nabla)$. Then we have $\deg(L_i)\geq 0$. Moreover, $\deg(F)\leq \deg(L_1)$ holds for all subbundles $F\subset L_1.$
\end{prop}

As we did in \S\ref{sect_odd}, we have an operator $\xi$ on the set of quasi gr-semistable  filtrations.
\begin{enumerate}
\item If
$$\Fil: 0= L_0 \subsetneq L_1 \subsetneq \cdots \subsetneq L_m \subsetneq L_m^\perp\subsetneq\cdots  \subsetneq L_1^\perp\subsetneq L_0^\perp= V$$
is of odd type, we define the  orthogonal Higgs sheaf $(\Gr_{\Fil}(V), \overline{\langle \, , \, \rangle},\theta)$, where $$\Gr_{\Fil}(V)=\mathop{\bigoplus}\limits_{i=1}^{m} \frac{L_i}{L_{i-1}}\oplus \frac{L_m^\perp}{L_m}\oplus \mathop{\bigoplus}\limits_{i=1}^{m} \frac{L_{m-i}^\perp}{L_{m-i+1}^\perp}.$$
If $(\Gr_{\Fil}(V),\theta)$ is not semistable, we have a maximal destabilizer $M_{\Fil}$ and by Lemma \ref{destabilizer}, we have
$$M_{\Fil}=\mathop{\bigoplus}\limits_{i=1}^{m} \frac{M_i}{L_{i-1}}\oplus \frac{M_{m+1}}{L_m}\oplus \mathop{\bigoplus}\limits_{i=1}^{m} \frac{N_{m-i+1}^\perp}{L_{m-i+1}^\perp},$$
where $L_{i-1}\subset M_i\subset N_i\subset L_i$ for $1\leq i\leq m$ and $M_{m+1}$ is isotropic. Based on the maximal destablizer, we construct a new filtration $\xi(\Fil)$ as follows:
$$\xi(\Fil):0=M_0\subset M_1\subset\cdots\subset M_{m+1}\subset  M_{m+1}^\perp\subset\cdots\subset M_1^\perp\subset M_0^\perp=V.$$
		
\item If $$\Fil: 0= L_0 \subsetneq L_1 \subsetneq \cdots\subsetneq L_{m-1} \subsetneq L_m \subsetneq L_{m-1}^\perp\subsetneq\cdots  \subsetneq L_1^\perp\subsetneq L_0^\perp= V$$ is of even type, we define $(\Gr_{\Fil}(V),\langle \ \rangle,\theta)$ in a similar way, where $$\Gr_{\Fil}(V)=\mathop{\bigoplus}\limits_{i=1}^{m} \frac{L_i}{L_{i-1}}\oplus \mathop{\bigoplus}\limits_{i=1}^{m} \frac{L_{m-i}^\perp}{L_{m-i+1}^\perp}.$$
If $(Gr_{\Fil}(V),\theta)$ is not semistable, we have the maximal destabilizer
$$M_{\Fil}=\mathop{\bigoplus}\limits_{i=1}^{m} \frac{M_i}{L_{i-1}}\oplus \mathop{\bigoplus}\limits_{i=1}^{m} \frac{N_{m-i+1}^\perp}{L_{m-i+1}^\perp},$$
where $L_{i-1}\subset M_i\subset N_i\subset L_i$ for $1\leq i\leq m$. Then, the new filtration $\xi(\Fil)$ is defined as
$$\xi(\Fil):0=M_0\subset M_1\subset\cdots\subset M_m \subset M_{m+1}\subset  M_m^\perp\subset\cdots\subset M_1^\perp\subset M_0^\perp=V,$$
where $M_{m+1}:=L_m$.
\end{enumerate}

\begin{lem}
Let $\Fil$ be a quasi gr-semistable filtration of odd type (resp. even type), then $\xi(\Fil)$ is  a Hodge filtration of odd type (resp. even type). Moreover, $\xi(\Fil)$ is also quasi gr-semistable of odd type (resp. even type).
\end{lem}

\begin{proof}
The proof of this lemma is similar to Lemma \ref{preserve}.
\end{proof}
	
\begin{lem}
Let $\Fil$ be a quasi gr-semistable filtration of $(V,\langle \, , \, \rangle,\nabla)$ in the even rank case. Assume that $(\Gr_{\Fil}(V),\theta_1)$ and $(\Gr_{\xi(\Fil)}(V),\theta_2)$ are not semistable. Let $M_{\Fil}$ (resp. $M_{\xi(\Fil)})$ be the maximal destabilizer of $(\Gr_{\Fil}(V),\theta_1)$ (resp. $(\Gr_{\xi(\Fil)}(V),\theta_2))$.  Then we have a natural morphism
$$\Phi: M_{\xi(\Fil)}\longrightarrow M_{\Fil}$$
such that
\begin{enumerate}
\item $\mu(M_{\xi(\Fil)})\leq \mu(M_{\Fil})$,

\item if $\mu(M_{\xi(\Fil)})= \mu(M_{\Fil})$, the morphism $\Phi$ is injective, and then we have $\rk(M_{\xi(\Fil)})\leq\rk(M_{\Fil})$,

\item if $\big(\mu(M_{\xi(\Fil)},\ \rk(M_{\xi(\Fil)})\big)= \big(\mu(M_{\Fil}),\rk(M_{\Fil})\big)$, then $\Phi$ is an isomorphism on an open subset $U \subset X$ with ${\rm codim}(X/U) \geq 2$.
\end{enumerate}
\end{lem}

\begin{proof}
This lemma is an analogue of Lemma \ref{decrease}.
\end{proof}
	
\begin{thm}\label{main thm even}
Let $(V,\langle \, , \, \rangle,\nabla)$ be semistable and quasi gr-semistable with even rank. There exsits a gr-semistable filtration on $(V,\langle \, , \, \rangle,\nabla)$.
\end{thm}

\begin{proof}
This theorem is an analogue of Theorem \ref{main thm}.
\end{proof}

Combining Proposition \ref{prop_ssregular_even} with Theorem \ref{main thm even}, we have the main result in the even rank case.
\begin{thm}\label{thm_even}
Let $(V,\langle \, , \, \rangle,\nabla)$ be a  semistable orthogonal/symplectic sheaf of even rank together with a $\lambda$-connection. Then $(V,\langle \, , \, \rangle,\nabla)$ is gr-semistable if and only if $(V,\langle \, , \, \rangle,\nabla)$ is quasi gr-semistable.
\end{thm}

\section{Classification of Quasi Gr-semistable Orthogonal $\lambda$-Connections in Small Ranks}\label{sect_small_rank}

In this section, $(V, \langle \, , \, \rangle, \nabla)$ is an orthogonal sheaf together with an orthogonal $\lambda$-connection $\nabla$. We make a careful discussion on the existence of quasi gr-semistable filtrations when $\rk (V) \leq 6$. We prove that $(V, \langle \, , \, \rangle, \nabla)$ is always quasi gr-semistable when $\rk(V) \leq 4$ (Proposition \ref{rank4}), and we classify quasi gr-semistable $\lambda$-connections when $\rk(V) = 5 \text{ and }6$ (Proposition \ref{irregular}, \ref{classfication5} and \ref{classfication6}). Moreover, Proposition \ref{irregular} gives a construction of $(V, \langle \, , \, \rangle, \nabla)$ which is not quasi gr-semistable when $\rk(V) \geq 5$.

\subsection{Case I: $\boldsymbol{\rk(V) \leq 4}$}
We start with a discussion of rank one isotropic saturated subsheaves $N\subset V$. Note that if a global section of $N$ vanishes on an open subset $U$, then it is trivial. For this reason, it is enough to work on an open subset $U \subset X$, on which $N|_U$ is locally free. Let $e \in N(U)$ be a local section. We have
$$\langle \nabla(e),e \rangle+\langle e,\nabla(e) \rangle=\lambda d\langle e,e \rangle=0,$$
which implies $\langle \nabla(e),e \rangle=0$. Therefore,
$$\nabla(N)\subset N^\perp\otimes\Omega_X(\log D)$$
for any isotropic saturated subsheaf $N$ of rank one. Moreover, this property does not hold in the symplectic case, i.e. $\langle \, , \, \rangle$ is skew-symmetric. Based on this fact, we have the following lemma about a special orthogonal filtration
$$0\subset N\subset N^\perp\subset V.$$	
	
\begin{lem}\label{criterion}
Given $(V,\langle \, , \, \rangle,\nabla)$, let $N$ be a rank one isotropic saturated subsheaf. The filtration
$$0\subset N\subset N^\perp\subset V$$
is quasi gr-semistable if and only if the following two conditions hold:
\begin{enumerate}
\item $\deg(N)\geq 0,$

\item if an isotropic saturated subsheaf $W$ containing N satisfies one of the following conditions:
\begin{enumerate}
\item $\deg(W)>2\deg(N);$

\item $\deg(N)<\deg(W)\leq2\deg(N),$ and $\nabla(N)\subset W^\perp\otimes \Omega_X(\log D);$

\item $0<\deg(W)\leq \deg(N),$ and $\nabla(N)\subset W\otimes \Omega_X(\log D),$
\end{enumerate}
then we must have $\nabla(W)\subset W^\perp\otimes \Omega_X(\log D).$
\end{enumerate}
\end{lem}
	
\begin{proof}	
This is a direct consequence of Example \ref{smallm}.
\end{proof}

\begin{lem}\label{cork1}
Given $(V,\langle \, , \, \rangle,\nabla)$,  let $N \subset W \subset V$ be isotropic saturated subsheaves of $V$ such that $\rk(N)=\rk(W)-1.$ Then, $\nabla(W)\subset W^\perp\otimes \Omega_X(\log D)$ if and only if $\nabla(N)\subset W^\perp\otimes \Omega_X(\log D)$.
\end{lem}
\begin{proof}
One direction is clear, and we only have to prove that if $\nabla(N)\subset W^\perp\otimes \Omega_X(\log D)$, we have $\nabla(W)\subset W^\perp\otimes \Omega_X(\log D)$. Since this is a local property, we work on an open subset $U \subset X$, on which $N$ is locally free. For convenience, we directly suppose that $$N=\mathcal{O}_X\{e_1, e_2,\dots e_n\},\quad W=\mathcal{O}_X\{e_1, e_2,\dots e_n, e_{n+1}\}.$$
Furthermore, $\nabla(W)\subset W^\perp\otimes \Omega_X(\log D)$ if and only if
$$\langle \nabla(W),W \rangle=0$$
holds. Therefore, it is equivalent to prove that  $\langle \nabla(e_i), e_j \rangle=0$ holds for all $1\leq i,j\leq n+1.$ Since $\nabla(N)\subset W^\perp\otimes \Omega_X(\log D)$, then $\langle \nabla(e_i), e_j \rangle=0$ holds for all $1\leq i\leq n$ and $1\leq j\leq n+1.$ Thus, we only have to check $\langle \nabla(e_{n+1}), e_i \rangle=0$ for $1\leq i\leq n+1$. The equality
$$\langle \nabla(e_{n+1}), e_{n+1} \rangle+\langle e_{n+1}, \nabla(e_{n+1}) \rangle=\lambda d\langle e_{n+1},e_{n+1} \rangle=0$$
implies that $\langle \nabla(e_{n+1}), e_i \rangle=0$ holds for $i= n+1$, and the equality $$\langle \nabla(e_{n+1}), e_i \rangle+\langle e_{n+1}, \nabla(e_i) \rangle=\lambda d\langle e_{n+1},e_i \rangle=0$$
and the condition $\langle e_{n+1}, \nabla(e_i) \rangle=0$ implies that $\langle \nabla(e_{n+1}), e_i \rangle=0$ holds for all $1\leq i\leq n$. This finishes the proof of this lemma.
\end{proof}

Lemma \ref{criterion} and Lemma \ref{cork1} give the following result:
\begin{lem}\label{criterionrk5}
Suppose that $\rk(V) \leq 5$. Let $N$ be a rank one isotropic saturated subsheaf. Then the filtration $$0\subset N\subset N^\perp\subset V$$
is a quasi gr-semistable filtration of $(V, \langle \, , \, \rangle, \nabla)$ if and only if the following two conditions hold:
\begin{enumerate}
\item $\deg(N)\geq 0,$

\item if $W$ is a rank two isotropic saturated subsheaf such that $N \subset W$ and $\deg(W)>2\deg(N)$, then we have $\nabla(W)\subset W^\perp\otimes \Omega_X(\log D).$
\end{enumerate}
\end{lem}

\begin{prop}\label{rank4}
Any $(V, \langle \, , \, \rangle, \nabla)$ with $\rk(V) \leq 4$ is quasi gr-semistable.
\end{prop}

\begin{proof}
It is trivial when $\rk(V)\leq3$, and then we only discuss the case $\rk(V)=4.$ If $(V, \langle \, , \, \rangle)$ is a semistable orthogonal sheaf, then the trivial filtration $$0\subset V$$
is quasi gr-semistable.

Now we assume that $(V, \langle \, , \, \rangle)$ is unstable and let $M$ be the maximal destabilizer. If $M$ is of rank one, the filtration
$$0\subset M\subset M^\perp\subset V$$
is quasi gr-semistable by Lemma \ref{criterionrk5} because there is no rank two isotropic subsheaf $W$ such that $\deg(W) > 2 \deg(M)$. Then, we suppose that $M$ is of rank two, i.e. it is Lagrangian. If there exists a rank one subsheaf $N$ such that
\begin{itemize}
    \item $\deg(N) \geq 0$ and $N \subset M$,
    \item $\nabla(N) \subset M\otimes \Omega_X(\log D)$,
\end{itemize}
then the filtration
$$0\subset N\subset M\subset N^\perp\subset V$$ is an even type quasi gr-semistable filtration, otherwise the filtration
$$0\subset M\subset V$$
is an even type quasi gr-semistable filtration. In conclusion, we find a quasi gr-semistable filtration for all cases. This finishes the proof.
\end{proof}
	

\subsection{Case II: $\boldsymbol{\rk(V) = 5}$}\label{subsect_rak_5}
Before we consider the case $\rk(V)=5$, we first construct a special orthogonal sheaf $(V, \langle \, , \, \rangle)$ together with an orthogonal $\lambda$-connection $\nabla$, which is not quasi gr-semistable.
	
\begin{prop}\label{irregular}
Let $(V, \langle \, , \, \rangle)$ be an unstable orthogonal sheaf of $\rk(V)\geq 5$. Suppose that
\begin{enumerate}
\item the maximal destabilizer  $M$ is of rank two,
\item $\frac{M^\perp}{M}$ is a stable orthogonal sheaf,
\item there exists a $\lambda$-cononection $\nabla$ satisfying
$$\nabla(M)\not \subset  M^\perp\otimes\Omega_X(\log D).$$
\end{enumerate}
Then $(V,\langle \, , \, \rangle,\nabla)$ is not quasi gr-semistable. Moreover, $(V, \langle \, , \, \rangle, \nabla)$ is semistable.
\end{prop}
	
\begin{proof}
We prove the first statement by contradiction, and assume that $(V,\langle \, , \, \rangle,\nabla)$ is quasi gr-semistable. Let
$$\Fil: 0= L_0 \subsetneq L_1 \subsetneq L_2 \cdots  \subsetneq L_1^\perp\subsetneq L_0^\perp= V$$
be a quasi gr-semistable filtration.

Since $\frac{V}{M^\perp}\cong M^\vee$ is a stable bundle of $\deg(\frac{V}{M^\perp})<0$ and $\mu_{\min}(L_1)\geq 0$ by Proposition \ref{positive}, the composition
$$L_1\hookrightarrow V\twoheadrightarrow \frac{V}{M^\perp}$$
is zero. That is $L_1\subset M^\perp$.

Now we consider the composition
$$L_1\hookrightarrow M^\perp\twoheadrightarrow \frac{M^\perp}{M}.$$
Note that this composition is not surjective since $L_1$ is isotropic while $\frac{M^\perp}{M}$ is not isotropic. Then this map is zero because $\mu_{\min}(L_1)\geq 0$ and $\frac{M^\perp}{M}$ is a stable orthogonal sheaf of degree $0$ by assumption. This implies that $L_1\subset M$.

If $L_1=M$, then we have
$$\nabla(M)\subset L_2\otimes \Omega_X(\log D)\subset M^\perp\otimes \Omega_X(\log D),$$
which contradicts the assumption that $\nabla(M)\not \subset  M^\perp\otimes\Omega_X(\log D)$. Therefore, $L_1$ is a proper subsheaf of $M$.

Now consider the natural morphism
$$h: L_2\longrightarrow \frac{L_1^\perp}{M^\perp}$$
based on the above argument that $L_1$ is a proper subsheaf of $M$. If $h=0$, then $L_2\subset M^\perp$, and we have $$\nabla(L_1)\subset L_2\otimes \Omega_X(\log D)\subset M^\perp\otimes \Omega_X(\log D).$$
This actually implies
$$\nabla(M) \subset  M^\perp\otimes\Omega_X(\log D)$$ by Lemma \ref{cork1}, and we get a contradiction again. On the other hand, if $h\neq 0$, then
$$\deg(\ker(h))\geq \deg(L_2)+\deg(M)-\deg(L_1)> \deg(L_2)+\deg(L_1).$$
This inequality contradicts Condition I. in Definition \ref{defn_gr_semi_odd}. In conclusion, $(V, \langle \, , \, \rangle, \nabla)$ is not quasi gr-semistable.

\hfill{\space}

For the second statement, we still prove it by contradiction. Suppose that $(V, \langle \, , \, \rangle, \nabla)$ is unstable. Denote by $W$ the maximal destabilizer of $(V, \langle \, , \, \rangle, \nabla)$. Clearly, $W$ is isotropic of $\deg(W)>0$ and $\nabla(W)\subset W\otimes \Omega_X(\log D).$
	
If $W$ is a rank one saturated subsheaf, we apply the same argument above for compositions
\begin{align*}
    W \hookrightarrow V \twoheadrightarrow \frac{V}{M^{\perp}}, \quad W \hookrightarrow M^\perp\twoheadrightarrow \frac{M^\perp}{M}
\end{align*}
and obtain $W \subset M$. Therefore,
$$\nabla(M)\subset M^\perp\otimes\Omega_X(\log D)$$ by Lemma \ref{cork1}, which contradicts the assumption in the statement of the proposition. If $\rk(W)\geq 2$, we have $W= M$. Since $W$ is $\nabla$-invariant, $M$ is also $\nabla$-invariant. This also gives a contradiction.

Finally, suppose that $W$ is unstable as an orthogonal sheaf. Denote by $N\subset W$ be the maximal destabilizer. With the same argument as above, we have $N\subset M$ and $W\subset M^\perp.$ Therefore, $$\nabla(N)\subset W\otimes \Omega_X(\log D)\subset M^\perp\otimes \Omega_X(\log D),$$
and this implies that $\nabla(M) \subset  M^\perp\otimes\Omega_X(\log D)$ by Lemma \ref{cork1}, which is a contradiction.

In conclusion, $(V, \langle \, , \, \rangle, \nabla)$ is semistable.
\end{proof}

\begin{examp}\label{exmp_idea}
In this example, we will give a way to construct a triple $(V, \langle \, , \, \rangle, \nabla)$ satisfying conditions in Proposition \ref{irregular}.

Let $X$ be a smooth projective curve with genus $g(X) \geq 2$. Let $M$ be a rank two stable bundle of positive degree, and there is a natural orthogonal structure on $M \oplus M^{\vee}$ and denote it by $\langle \, , \, \rangle_M$. Let $N$ be an arbitrary stable orthogonal bundle on $X$ and the orthogonal structure of $N$ is denoted by $\langle \, , \, \rangle_N$. Define
$$V=M\oplus M^\vee\oplus N.$$
We have a natural orthogonal structure $\langle \, , \, \rangle$ on $V$ induced by $\langle \, , \, \rangle_M$ and $\langle \, , \, \rangle_N$.

Let $\nabla_0: M\longrightarrow M\otimes\Omega_X(\log D)$ be a logarithmic $\lambda$-connection on $M$, and $\nabla_0^\vee: M^\vee \rightarrow M^\vee \otimes \Omega_X(\log D)$ its dual logarithmic connection of $\nabla_0$. Then,
\begin{align*}
\begin{pmatrix}
    \nabla_0 & \\
    & -\nabla^\vee_0
\end{pmatrix}: M \oplus M^\vee \rightarrow (M \oplus M^\vee) \otimes \Omega_X(\log D)
\end{align*}
is an orthogonal $\lambda$-connection on $M \oplus M^\vee$. We also choose $\theta_1: M \rightarrow M^\vee \otimes \Omega_X(\log D)$ a nontrivial $\mathcal{O}_X$-linear morphism such that
$$\langle \theta_1(a),b \rangle_M + \langle a,\theta_1(b) \rangle_M = 0$$
for all local sections $a,b$ of $M$. We may always find such a section by enlarging the support of $D$. Clearly,
\begin{align*}
\begin{pmatrix}
    \nabla_0 & \\
    \theta_1 & -\nabla^\vee_0
\end{pmatrix}: M \oplus M^\vee \rightarrow (M \oplus M^\vee) \otimes \Omega_X(\log D)
\end{align*}
is an orthogonal $\lambda$-connection. Choosing a $\lambda$-connection $\nabla_N$ on $N$, we get an orthogonal $\lambda$-connection
\begin{align*}
\nabla:=
\begin{pmatrix}
    \nabla_0 & & \\
    \theta_1 & -\nabla^\vee_0 & \\
    & & \nabla_N
\end{pmatrix}: V \rightarrow V \otimes \Omega_X(\log D)
\end{align*}
on $(V, \langle \, , \, \rangle)$. From the above construction, we have
$$ \nabla(M)\not\subset M^\perp\otimes\Omega_X(\log D).$$
Thus, we obtain an orthogonal bundle $(V, \langle \, , \, \rangle)$ of rank larger than five with an orthogonal $\lambda$-connection $\nabla$ satisfying all conditions in Proposition \ref{irregular}.
\end{examp}

\begin{examp}\label{exmp_const}
In the above example, we give an approach to construct a triple $(V, \langle \, , \, \rangle, \nabla)$, which is not quasi gr-semistable. Following this approach, we provide a counterexample for Simpson's question (Question \ref{quest_simp}) in positive characteristic.

Let $X$ be a smooth projective curve over $k$ with genus $g$ such that
\begin{align*}
    3\leq p < g-1,
\end{align*}
where $p$ is the characteristic of $k$. There exists a stable bundle $M'$ on $X$ such that $\rk(M') = 2$ and $\deg(M')=1$. Consider the Frobenius pullback $M:=F^*_X(M')$, of which $\rk(M)=2$ and $\deg(M)=p$. Although $M$ may not be semistable \cite{Gi73}, we can choose a specific $M'$ such that $M$ is semistable. Since the degree $p$ and the rank $2$ are coprime, the semistable bundle $M$ is stable. Let $\nabla_0: M \rightarrow M \otimes \Omega_X$ be the canonical connection. Next, we consider the sheaf $(\wedge^2 M^{\vee}) \otimes \Omega_X$. By Riemann--Roch formula, we have
\begin{align*}
    h^0((\wedge^2 M^\vee) \otimes \Omega_X) & \geq \deg( (\wedge^2 M^\vee) \otimes \Omega_X ) + (1-g) \\
    & = (-p+2g-2) + (1-g) \\
    & = -p + (g-1) > 0
\end{align*}
Therefore, there exists a nontrivial section of $(\wedge^2 M^{\vee}) \otimes \Omega_X$. Following the construction in Example \ref{exmp_idea}, we construct an orthogonal bundle
\begin{align*}
    V:= M \oplus M^{\vee} \oplus \mathcal{O}_X
\end{align*}
of rank five with orthogonal structure $\langle \, , \, \rangle$. Then, the nontrivial section of $(\wedge^2 M^{\vee}) \otimes \Omega_X$ induces a nontrivial $\mathcal{O}_X$-linear morphism
\begin{align*}
    \theta_1: M \rightarrow M^{\vee} \otimes \Omega_X
\end{align*}
such that for local sections $a,b$ of $M$,
$$\langle \theta_1(a),b \rangle_M + \langle a,\theta_1(b) \rangle_M = 0.$$
By Example \ref{exmp_idea}, the orthogonal connection
\begin{align*}
\nabla:=
\begin{pmatrix}
    \nabla_0 & & \\
    \theta_1 & -\nabla^\vee_0 & \\
    & & d
\end{pmatrix}: V \rightarrow V \otimes \Omega_X \ ,
\end{align*}
is not quasi gr-semistable, where $d: \mathcal{O}_X \rightarrow \mathcal{O}_X$ is the exterior differential. Under the equivalence of gr-semistability and quasi gr-semistability (Theorem \ref{thm_gr_and_quasi_gr_odd}), we conclude that we find an example in rank five which is not gr-semistable.
\end{examp}

Proposition \ref{irregular} gives the construction of a special class of $(V, \langle \, ,\, \rangle, \nabla)$, which is not quasi gr-semistable. Furthermore, we can show that if $(V, \langle \, ,\, \rangle, \nabla)$ is not quasi gr-semistable and if $\rk(V)=5 \text{ and } 6$, then it must satisfy the conditions in Proposition \ref{irregular}. In the next proposition, we shall discuss the case $\rk(V)=5$ first.
	
\begin{prop}\label{classfication5}
Suppose that $\rk(V)=5$. A triple $(V,\langle \, , \, \rangle,\nabla)$ is not quasi gr-semistable if and only if $(V, \langle \, , \, \rangle)$ is an unstable orthogonal sheaf such that the maximal destabilizer $M$ is stable of rank two and satisfies
$$ \nabla(M)\not\subset M^\perp\otimes\Omega_X(\log D).$$
\end{prop}
	
\begin{proof}		
One direction is given by Proposition \ref{irregular}. We consider the other direction and assume that $(V,\langle \, , \, \rangle,\nabla)$ is not quasi gr-semistable. Clearly, $(V, \langle \, , \, \rangle)$ is unstable, otherwise the trivial filtration $$0\subset V$$ is quasi gr-semistable. Now denote by $M$ the maximal destabilizer of $V$. If $M$ is not a stable rank two bundle, there exists saturated subsheaf $N \subset M$ of rank one such that $\deg(N)=\mu(M).$ Then, the filtration
$$0\subset N\subset N^\perp\subset V$$
is quasi gr-semistable by Lemma \ref{criterionrk5}. Therefore, $M$ is stable of rank two. Furthemore, we have
$$\nabla(M)\not\subset M^\perp\otimes\Omega_X(\log D),$$ otherwise the filtration $$0\subset M \subset M^\perp\subset V$$ is quasi gr-semistable.
		
\end{proof}

\subsection{Case III: $\boldsymbol{\rk(V) = 6}$}
For the rank $6$ case, we first prove two lemmas.

\begin{lem}\label{degreeincreasing}
Suppose that $(V,\langle \, , \, \rangle,\nabla)$ is not quasi gr-semistable of rank $6$. Let $$\Fil_{LW}:0\subset L\subset W\subset L^\perp\subset V$$ be an orthogonal Hodge filtration such that
\begin{itemize}
    \item $\rk(L)=1, \ \rk(W)=3$ and $\deg(L)>0, \ \deg(W)>0$,
    \item $\deg(L)+\deg(W)>\mu_{\max}(V)$,
    \item $2\deg(L)+\deg(W)>2\mu_{\max}(V)$.
\end{itemize}
Then there exists a rank three isotropic subsheaf $W^\prime$ containing $L$ and satisfying one of the following conditions:
\begin{enumerate}
    \item $\deg(W^\prime)>\max\{2\deg(L),\deg(W)\}$;
    \item $\deg(W^\prime)>\deg(W)$, and the filtration $$\Fil_{LW}:0\subset L\subset W^\prime\subset L^\perp\subset V$$ is an orthogonal Hodge filtration;
    \item $\deg(W^\prime)>2\deg(L),$ and there exists a rank two subbundle $P_2$ of $W$ such that $$2\deg(P_2)>\deg(W)+2\deg(L),\ \nabla(P_2)\subset P_2^\perp\otimes\Omega_X(\log D),\ \nabla(P_2)\not\subset W\otimes\Omega_X(\log D).$$
\end{enumerate}
\end{lem}
\begin{proof}
Since $(V,\langle \, , \, \rangle,\nabla)$ is not quasi gr-semistable, the filtration $\Fil_{LW}$ does not satisfy Condition I or II in Definition \ref{defn_gr_semi_even}. If $\Fil_{LW}$ does not satisfy Condition I, there exists a rank two isotropic sheaf $S$ between $L$ and $W$ such that
\begin{align*}
\deg(S)>\deg(L)+\deg(W)
\end{align*}
or
\begin{align*}
\deg(S)>\deg(W) \text{ and } \nabla(L)\subset S\otimes\Omega_X(\log D)
\end{align*}

\hfill{\space}

For the \emph{first case} $\deg(S)>\deg(L)+\deg(W)$, we consider
$$\Fil_S:0\subset S\subset S^\perp\subset V,$$
which is not quasi gr-semistable. Clearly, $\mu_{\min}(S)>0 $ because
\begin{align*}
    \deg(S)>\deg(L)+\deg(W)>\mu_{\max}(V).
\end{align*}
Then there exist isotropic subsheaves
$$0\subset T\subset U\subset S\subset W^\prime $$
such that $\rk(U)\leq1,\ \rk(W^\prime)=3$ and
\begin{align*}
    \nabla(U )\subset W^\prime\otimes\Omega_X(\log D),\  \deg(W^\prime)>2\deg(S)-\deg(U)-\deg(T).
\end{align*}
Then,
\begin{equation*}
\begin{split}
   \deg(W^\prime)&> 2\deg(S)-\deg(U)-\deg(T)\\
   &>\deg(W)+2\deg(L)+\deg(W)-\deg(U)-\deg(T)\\
   &>\deg(W)+2\mu_{\max}(V)-\mu_{\max}(V)-\mu_{\max}(V)=\deg(W)
\end{split}
\end{equation*}
Moreover, $U=L$ or $\deg(U)\leq \deg(S)-\deg(L)$. If $U=L$, it is clear
\begin{align*}
    \nabla(L)\subset W^\prime\otimes\Omega_X(\log D).
\end{align*}
If $\deg(U)\leq \deg(S)-\deg(L),$ we have
\begin{align*}
    \deg(W^\prime)>2\deg(S)-2(\deg(S)-\deg(L))=2\deg(L).
\end{align*}

\hfill{\space}

For the \emph{second case} that $\deg(S)>\deg(W)$ and $\nabla(L)\subset S\otimes\Omega_X(\log D)$, consider the filtration
$$\Fil_{LS}:0\subset L\subset S\subset S^\perp\subset L^\perp\subset V,$$
and there exists a rank three isotropic subsheaf $W^\prime$ containing $L$ such that $\deg(W^\prime)>2\deg(S)$ and
\begin{align*}
    \nabla(L)\subset W^\prime\otimes\Omega_X(\log D).
\end{align*}
Therefore,
\begin{align*}
    \deg(W^\prime)>2\deg(S)>2\deg(W)>\deg(W).
\end{align*}

\hfill{\space}

If $\Fil_{LW}$ does not satisfy Condition II, then there exists a Hodge filtration
$$P_1\subset P_2\subset  P_2^\perp\subset Q_1^\perp\subset V$$
such that $$0\subset P_1\subset Q_1\subset L  \subset P_2\subset W,$$  $$\rk(P_2)=2,\ 2\deg(P_2)>\deg(W)+2\deg(L)-\deg(P_1)-\deg(Q_1),$$ and $$\nabla(P_2)\not\subset W\otimes\Omega_X(\log D).$$ Note that $Q_1=0$ or $L.$ If $Q_1=0,$ then $P_1=0$. Thus,
\begin{align*}
    \deg(P_2)>\dfrac{1}{2}(\deg(W)+2\deg(L))>\mu_{\max}(V).
\end{align*}
With a similar argument as the proof in the \emph{first case}, such a subsheaf $W^\prime$ exists by considering the filtration
\begin{align*}
    \Fil_{P_2}:0\subset P_2\subset P_2^\perp\subset V.
\end{align*}
If $Q_1=L,$  a similar argument as in the \emph{second case} will produce such a subsheaf $W^\prime$ by considering
\begin{align*}
    \Fil_{LP_2}:0\subset L\subset P_2\subset P_2^\perp\subset L^\perp\subset V.
\end{align*}
This finishes the proof of this lemma.
\end{proof}

\begin{lem}\label{lagrange}
Suppose that $(V,\langle \, , \, \rangle,\nabla)$ is not quasi gr-semistable and is of rank six. Define $$\alpha(V):=\max_{M}\ \{\deg(M)\ |\ M \text{ is a rank three isotropic subsheaf of $V$ }\}.$$
Then $\alpha(V)<2\mu_{\max}(V).$
\end{lem}

\begin{proof}
We still prove this lemma by contradiction, and assume that $\alpha(V)\geq2\mu_{\max}(V).$ Let $M$ be a rank three isotropic subsheaf of $V$ such that $\deg(M)=\alpha(V).$  Then we have $\mu_{\min}(M)\geq0$ since $\deg(M)\geq 2\mu_{\max}(V).$ This implies that the filtration
$$\Fil_{M}:0\subset M\subset V,$$
which is not quasi gr-semistable, satisfies Condition I. There exist saturated subsheaves $S\subset T\subset M$ such that
\begin{align*}
    \deg(T)+\deg(S)>\deg(M),\
    \nabla(S)\subset T^\perp\otimes\Omega_X(\log D), \
    \nabla(S)\not\subset M\otimes\Omega_X(\log D).
\end{align*}
Then we have $\rk(T)=2$, otherwise
$$\deg(T)+\deg(S)\leq \mu_{\max}(V)+\mu_{\max}(V)=2\mu_{\max}(V)\leq \deg(M),$$
which is a contradiction. In summary, there exists a rank two subsheaf $T$ of $M$ such that $$\deg(T)>\frac{\deg(M)}{2}\geq \mu_{\max}(V),\ \nabla(T)\subset T^\perp\otimes\Omega_X(\log D),$$ and $$\nabla(T)\not\subset M\otimes\Omega_X(\log D).$$
Moreover, $\mu_{\min}(T)\geq 0$.

\hfill{\space}

Now we choose such a subsheaf $T$ with maximal degree and consider the orthogonal Hodge filtration $$\Fil_{T}:0\subset T\subset T^\perp\subset V.$$
Since it is not quasi gr-semistable, there exist saturated subsheaves $S_1\subset T_1\subset T\subset S_2$ such that $S_2$ is isotropic of rank three and
$$\nabla(T_1)\subset S_2\otimes\Omega_X(\log D),\ \nabla(S_2)\not\subset S_2\otimes\Omega_X(\log D),$$
and
$$\deg(S_2)>2\deg(T)-\deg(T_1)-\deg(S_1).$$
Since $\nabla(S_2)\not\subset S_2\otimes\Omega_X(\log D)$, we have $\rk(T_1)\leq1$ by Lemma \ref{cork1}. If $T_1=0,$ clearly
$$\deg(S_2)>2\deg(T)\geq \deg(M).$$
This contradicts our choice of $M$. Therefore, we have $\rk(T_1)=1.$ Now if $\rk(T_1)=1$, we get the following inequality
\begin{equation*}
\begin{split}
\deg(\frac{S_2}{T})-\deg(\frac{T^\perp}{M})&=\deg(M)+\deg(S_2)-2\deg(T)\\
&>\deg(M)-\deg(T_1)-\deg(S_1)\\
&\ge2\mu_{\max}(V)-\mu_{\max}(V)-\mu_{\max}(V)=0,
\end{split}
\end{equation*}
which implies $S_2=M$ by the natural morphism $\frac{S_2}{T}\longrightarrow \frac{T^\perp}{M}.$ Moreover, we have $\deg(T_1)>0.$

\hfill{\space}

Then, we consider the orthogonal Hodge filtration $$\Fil_{T_1M}:0\subset T_1\subset M\subset T_1^\perp \subset V.$$
By Lemma \ref{degreeincreasing}, there exists a rank two saturated subsheaf $P_2 \subset M$ such that $$2\deg(P_2)>\deg(M)+2\deg(T_1),\ \nabla(P_2)\subset P_2^\perp\otimes\Omega_X(\log D),\ \nabla(P_2)\not\subset M\otimes\Omega_X(\log D).$$
Then,
\begin{equation*}
	\begin{split}
		2\deg(P_2)&>\deg(M)+2\deg(T_1)\\
		&>2\deg(T)-\deg(T_1)-\deg(S_1)+2\deg(T_1)\\
		&=2\deg(T)+\deg(T_1)-\deg(S_1)\\
		&\geq 2\deg(T)
	\end{split}.
\end{equation*}
This contradicts our choice of $\deg(T)$, which is of maximal degree. In conclusion, we have $\alpha(V)<2\mu_{\max}(V).$
\end{proof}

\begin{prop}\label{classfication6}
Let $(V,\langle \, , \, \rangle,\nabla)$ be an orthogonal sheaf of rank six together with an orthogonal $\lambda$-connection. Then $(V,\langle \ \rangle,\nabla)$ is not quasi gr-semistable if and only if $V$ is an unstable orthogonal sheaf such that
\begin{enumerate}
    \item the maximal destabilizer $M$ is a stable rank two saturated subsheaf,
    \item $\frac{M^\perp}{M}$ is stable or $\frac{M^\perp}{M}$ is a direct sum of two nonisomorphic orthogonal torsion free sheaves of rank one,
    \item $\nabla(M)\not\subset M^\perp\otimes\Omega_X(\log D)$.
\end{enumerate}
\end{prop}

\begin{proof}
One direction is exactly proven by Proposition \ref{irregular}. Now we assume that $(V,\langle \ \rangle,\nabla)$ is not quasi gr-semistable. Clearly, the trivial filtration $0 \subset V$ is not quasi gr-semistable, and then $V$ is unstable. We prove the other three statements by contradiction.

\hfill{\space}

Let $M$ be the maximal destabilizer of $V$. Lemma \ref{lagrange} shows that $\rk(M) \leq 2$. If $M$ is not stable of rank two, then there exists a saturated subsheaf $N \subset M$ of rank one such that $\deg(N)=\mu(M)=\mu_{\max}(V).$ Now consider the  orthogonal Hodge filtration
$$\Fil_{N}:0\subset N\subset N^\perp\subset V.$$
Since it is not quasi gr-semistable, there exists a Hodge filtration
$$R_1\subset R_2\subset  U_1^\perp\subset V$$
such that $$0\subset R_1\subset U_1\subset N  \subset R_2$$
with $\rk(R_2)=2$ or $3$,
$$\deg(R_2)>2\deg(N)-\deg(R_1)-\deg(U_1),$$
and
$$\nabla(R_2)\not\subset R_2^\perp\otimes\Omega_X(\log D).$$
We first suppose $\rk(R_2)=2$. Then, the condition $\nabla(R_2)\not\subset R_2^\perp\otimes\Omega_X(\log D)$ and Lemma \ref{cork1} implies that $R_1 = U_1 = 0$.  Then,
$$\deg(R_2)>2\deg(N)=2\mu_{\max}(V),$$
and this is a contradiction. Thus, $\rk(R_2)=3$. In this case, $U_1=0$ or $U_1 = N$. With a similar argument as above, Lemma \ref{lagrange} implies $U_1=N$. In summary, there exists a rank three isotropic subsheaf $R_2$ such that $$\deg(R_2)>0,\ \nabla(N)\subset R_2 \otimes \Omega_X(\log D),\  \nabla(R_2)\not\subset R_2 \otimes \Omega_X(\log D).$$ Now we choose such a subsheaf $R_2$ with maximal degree and  consider the orthogonal Hodge filtration
$$\Fil_{NR_2}:0\subset N\subset R_2\subset N^\perp \subset V.$$
By Lemma \ref{degreeincreasing}, there exists an isotropic subsheaf $W^\prime$ of rank three such that $\deg(W^\prime)>2\deg(N)=2\mu_{\max}(V)$. However, this contradicts Lemma \ref{lagrange}. Therefore, the maximal destabilizer $M$ of $V$ is stable of rank two.

\hfill{\space}

For the second statement, suppose that $\frac{M^\perp}{M}$ is not stable. Then there exists a rank three isotropic subsheaf $Z$ contain $M$ such that $\deg(\frac{Z}{M})\geq 0.$ We have $\deg(Z)\geq \deg(M)=2\mu_{\max}(V)$, which contradicts Lemma \ref{lagrange}. Therefore, $\frac{M^\perp}{M}$ is stable.

\hfill{\space}

For the last statement, if $ \nabla(M)\subset M^\perp\otimes\Omega_X(\log D),$ we consider the Hodge filtration $$\Fil_{M}:0\subset M\subset M^\perp\subset V,$$
which is not quasi gr-semistable. Then there exist saturated subsheaves $S_1\subset T_1\subset M\subset S_2$ such that $S_2$ is isotropic of rank three and
$$\nabla(T_1)\subset S_2\otimes\Omega_X(\log D),\ \nabla(S_2)\not\subset S_2\otimes\Omega_X(\log D),$$
and
$$\deg(S_2)>2\deg(M)-\deg(T_1)-\deg(S_1).$$
Lemma \ref{cork1} implies $\rk(T_1)\leq1$. Then, 
\begin{align*}
    \deg(S_2) & >2\deg(M)-\deg(T_1)-\deg(S_1) \\
    & \geq 4\mu_{\max}(V)-\mu_{\max}(V)-\mu_{\max}(V)=2\mu_{\max}(V),
\end{align*}
and this contradicts Lemma \ref{lagrange}. Therefore, we have $ \nabla(M)\not\subset M^\perp\otimes\Omega_X(\log D).$
\end{proof}
	
Now apply the above results to algebraic curves. For elliptic curves, structure of vector bundles is clear \cite{A57,T93}. We have the following corollaries.
	
\begin{cor}
Let $X$ be an elliptic curve. Let $(V,\langle \, , \, \rangle,\nabla)$ be an orthogonal sheaf of rank five together with an orthogonal $\lambda$-connection on $X$. Then $(V,\langle \, , \, \rangle,\nabla)$ is not quasi gr-semistable if and only if $$V=M\oplus M^\vee\oplus N,$$  $$\ \nabla(M)\not\subset (M\oplus N)\otimes\Omega_X(\log D),$$
where $M$ is a rank two stable bundle of odd degree and $N$ is a line bundle such that $N^{\otimes 2}\cong \mathcal{O}_X.$
\end{cor}

\begin{cor}\label{elliptic6}
Let $X$ be an elliptic curve. Let $(V,\langle \, , \, \rangle,\nabla)$ be an orthogonal sheaf of rank six together with an orthogonal $\lambda$-connection on $X$. Then $(V,\langle \, , \, \rangle,\nabla)$ is not quasi gr-semistable if and only if
$$V=M\oplus M^\vee\oplus N_1\oplus N_2,$$  $$\ \nabla(M)\not\subset (M\oplus N_1\oplus N_2)\otimes\Omega_X(\log D),$$
where $M$ is a rank two stable bundle of odd degree, and $N_1,N_2$ are nonisomorphic line bundles such that $N_i^{\otimes 2}\cong \mathcal{O}_X$ for $i=1,2.$
\end{cor}
	
For projective line $\mathbb{P}^1_k$, we have the following corollary.	
\begin{cor}\label{proline5}
Let $(V,\langle \, , \, \rangle,\nabla)$ be an orthogonal bundle on $\mathbb{P}^1_k$ together with an orthogonal $\lambda$-connection. If $\rk(V) \leq 6$, then it is quasi gr-semistable.
\end{cor}

\begin{proof}
Note that there is no stable bundle of rank two on $\mathbb{P}^1_k$. This corollary is a direct consequence of Proposition \ref{classfication5} and \ref{classfication6}.
\end{proof}
	
Based on Corollary \ref{proline5}, we intend to believe the truth of the following question.

\begin{ques}
Is any orthogonal bundle with an orthogonal $\lambda$-connection on the projective line $\P^1_k$  quasi gr-semistable and hence gr-semistable?
\end{ques}\label{prolineconj}

\vspace{2mm}
{\bf Acknowledgement.} The first and third named authors are partially supported by the National Key R and D Program of China 2020YFA0713100, CAS Project for Young Scientists in Basic Research Grant No. YSBR-032. The second named author is partially supported by National Key R and D Program of China 2022YFA1006600.

\bibliographystyle{plain}

\bigskip
\noindent\small{\textsc{School of Mathematical Sciences, University of Science and Technology of China, \\ Hefei, 230026, China}} \\

\noindent\small{\textsc{Yau Mathematical Science Center, Tsinghua University, \\
Beijing, 100084, China}}\\
\emph{E-mail address}:  \texttt{msheng@mail.tsinghua.edu.cn}

\bigskip
\noindent\small{\textsc{Department of Mathematics, South China University of Technology, \\
Guangzhou, 510641, China}\\
\emph{E-mail address}:  \texttt{hsun71275@scut.edu.cn}

\bigskip
\noindent\small{\textsc{Yau Mathematical Science Center, Tsinghua University, \\
Beijing, 100084, China}\\
\emph{E-mail address}:  \texttt{jianpw@mail.tsinghua.edu.cn}	
\end{document}